\documentclass[reqno,centertags,12pt]{amsart}
\usepackage{amsmath,amsthm,amscd,amssymb,latexsym,verbatim,color, amssymb,amsfonts,gensymb, esint}

\usepackage{graphicx,epsf,cite}

-\marginparwidth \oddsidemargin -4mm \evensidemargin \oddsidemargin

\newtheorem{theorem}{Theorem}[section]

\newtheorem{proposition}[theorem]{Proposition}
\newtheorem{lemma}[theorem]{Lemma}
\newtheorem{corollary}[theorem]{Corollary}
\theoremstyle{definition}
\newtheorem{definition}[theorem]{Definition}

\theoremstyle{remark}

\definecolor{purple}{rgb}{0.65, 0, 1}
\definecolor{green}{rgb}{0, 1, 0}
\definecolor{orange}{rgb}{1,.5,0}
\definecolor{gray}{rgb}{0.8,0.8,0.8}

\newcounter{smalllist}

\definecolor{purple}{rgb}{0.65, 0, 1}
\definecolor{pink}{rgb}{1, 0, 0.9}
\definecolor{orange}{rgb}{1,.5,0}

\DeclareMathOperator*{\dist}{dist}

\allowdisplaybreaks
\numberwithin{equation}{section}

\newcommand{\lb}{\label}

\newcommand{\supp}{\text{\rm{supp}}}

\newcommand{\beq}{\begin{equation}}
\newcommand{\eeq}{\end{equation}}

\newcommand{\bal}{\begin{align}}
\newcommand{\eal}{\end{align}}
\newcommand{\bals}{\begin{align*}}
\newcommand{\eals}{\end{align*}}

\newcommand{\bbR}{{\mathbb{R}}}

\newcommand{\bbT}{{\mathbb{T}}}

\newcommand{\eps}{\varepsilon}
\newcommand{\del}{\delta}
\newcommand{\tht}{\theta}

\newcommand{\til}{\tilde}

\begin{document}
\title[Mixing and Un-mixing]
{Mixing and Un-mixing by Incompressible Flows}

\author{Yao Yao and Andrej Zlato\v s}

\address{\noindent Department of Mathematics \\ University of Wisconsin \\ Madison, WI 53706, USA 
}


\thanks{July 14, 2014.}

\begin{abstract}
We consider the questions of efficient mixing and un-mixing by incompressible flows which satisfy periodic, no-flow, or no-slip boundary conditions on a square.  Under the uniform-in-time constraint $\|\nabla u(\cdot,t)\|_p\le 1$ we show that any function can be mixed to scale $\eps$ in time $O(|\log\eps|^{1+\nu_p})$, with $\nu_p=0$ for $p<\tfrac{3+\sqrt 5}2$ and $\nu_p\le \tfrac 13$ for $p\ge \tfrac{3+\sqrt 5}2$. 
Known lower bounds show that this rate is optimal for $p\in(1,\tfrac{3+\sqrt 5}2)$.  We also show that any set which is mixed to scale $\eps$ but not much more than that can be un-mixed to a rectangle of the same area (up to a small error) in time $O(|\log\eps|^{2-1/p})$.  Both results hold with scale-independent finite times if the constraint on the flow is changed to $\|u(\cdot,t)\|_{\dot W^{s,p}}\le 1$ with some $s<1$.  The constants in all our results are independent of the mixed functions and sets.
\end{abstract}

\maketitle

\section{Introduction and Main Results} \lb{S1}

Mixing of substances by flows and processes involving it are ubiquitous in nature.  In the absence of diffusion, or when diffusion acts at time scales much longer than the flow and thus can be neglected in short and medium terms, the basic model for mixing of {\it passive scalars} (i.e., with no feedback of the mixed substance on the mixing flow) is the transport equation
\beq\lb{1.1}
\rho_t + u \cdot \nabla \rho = 0,
\eeq
with initial condition $\rho(\cdot, 0) = \rho_0$.  Here $\rho:Q\times\bbR^+\to\bbR$ is the mixed scalar (e.g., density of particles of a substance in a liquid), with $Q\subseteq\bbR^d$ the physical domain and $\rho_0:Q\to\bbR$, while $u:Q\times\bbR^+\to\bbR^d$ is the mixing flow.  Of particular interest in real-world applications are {\it incompressible flows} (with $\nabla\cdot u=0$) and questions of their mixing efficiency.  A natural (and central) problem in this direction is how well a given initial density $\rho_0$ can be mixed by incompressible flows satisfying some physically relevant quantitative constraints.  

For the sake of transparency, we will consider here the case of a square  $Q = (0,1)^2\subseteq\bbR^2$, with either the {\it no-flow boundary condition} $u\cdot n = 0$ on $\partial Q\times\bbR^+$ (where $n$ is the unit outer normal to $Q$), or the {\it no-slip boundary condition} $u=0$ on $\partial Q\times\bbR^+$, or the {\it periodic boundary condition} $\rho(0,r,t)=\rho(1,r,t)$ and $\rho(r,0,t)=\rho(r,1,t)$ for all $(r,t)\in (0,1)\times\bbR^+$ (when $Q$ becomes the torus $\bbT^2$).   We will also assume that $\rho_0 \in L^\infty(Q)$ (so $\|\rho(\cdot,t)\|_\infty= \|\rho_0\|_\infty$ for each $t>0$) and is mean-zero on $Q$ (i.e., $\int_Q \rho_0 \, dxdy = 0$).
Obviously, the latter is not essential since changing $\rho_0$ by a constant only changes $\rho$ by the same constant.  The restriction to two dimensions is also not essential, as our proofs easily extend to higher dimensions.

To quantify the mixing efficiency of flows, one needs to define a suitable measure of mixing (see, e.g., the review \cite{T}).  While in the case of diffusive mixing this may be done in terms of global quantities, such as the decay of $L^p$ norms of (a mean-zero) $\rho$ in time \cite{CKRZ, DT, STD, TDG}, solutions of \eqref{1.1} have a constant-in-time distribution function, so we need to look at small scale variations of $\rho$ instead.  
In the present paper we will consider the following natural definition, in which $\fint_A fdxdy=|A|^{-1}\int_A fdxdy$ is the average of the function $f$ over a set $A$.

\begin{definition} \lb{D.1.1}
Let $f \in L^\infty(Q)$ be mean-zero on $Q$ and let $\kappa,\eps\in(0,\tfrac 12]$.  We say that $f$ is \emph{$\kappa$-mixed to scale $\eps$} if for each $(x_0,y_0)\in Q$,
\[
\left|\fint_{B_\eps(x_0,y_0)\cap Q} f(x,y) dxdy\right| \leq \kappa \|f\|_{\infty}.
\]
If now $\rho_0 \in L^\infty(Q)$ is mean-zero, we say that an incompressible flow $u:Q \times \bbR^+ \to \mathbb{R}^2$ \emph{$\kappa$-mixes  $\rho_0$ to scale $\eps$ by time $\tau$} if $\rho(\cdot, \tau)$ is $\kappa$-mixed to scale $\eps$, where $\rho$ solves \eqref{1.1} with $\rho(\cdot, 0) = \rho_0$. 
\end{definition}

{\it Remark.}  Another natural definition of mixing that has been used recently is in terms of the $H^{-1}$-norm \cite{ACM,IKX, LLNMD, LTD, S} 
of $f$.
(Other $H^{-s}$ norms \cite{MMP} or the Wasserstein distance of $f_+$ and $f_-$ \cite{BOS, OSS, S, S2} have also been used.)
In this case there is no $\kappa$ and the mixing scale is given by $\|f\|_{H^{-1}}\|f\|_{\infty}^{-1}$. 
We discuss the relation of this definition to Definition~\ref{D.1.1} and our main results after Corollary \ref{C.1.5} below.

The motivation for  Definition \ref{D.1.1} comes from a paper by Bressan \cite{Bressan}, whose definition is a special case of ours.  He considered the case $Q=\bbT^2$ (i.e., periodic boundary conditions), $\kappa= \tfrac 13$, and $\rho_0=\chi_{(0,1/2)\times(0,1)} - \chi_{(1/2,1)\times(0,1)}$,  and conjectured that if an incompressible flow $u$ $\tfrac 13$-mixes $\rho_0$ to some scale $\eps\in(0,\tfrac 12]$ in time $\tau$, then
\[
\int_0^\tau \|\nabla u(\cdot, t)\|_1 dt \ge C |\log\eps|
\]
(with some $\eps$-independent $C<\infty$).  Or, equivalently (after an appropriate change of the time variable, as in the proof of Theorems \ref{T.1.1}--\ref{T.1.3} below), that there is $C<\infty$ such that if an incompressible flow $u$ satisfies
\beq \lb{1.2}
\sup_{t>0} \|\nabla u(\cdot, t)\|_1 \le 1
\eeq
and $\tfrac 13$-mixes $\rho_0$ to some scale $\eps\in(0,\tfrac 12]$ in time $\tau$, then $\tau\ge C|\log\eps|$.  One should think of  $\|\nabla u(\cdot, t)\|_1$ as an instantaneous cost of the mixing at time $t$, and  as we note 
in Remark 2 after Corollary \ref{C.1.5} 
below, the  critical order of derivatives of $u$ here is indeed 1. 

The above {\it rearrangement cost conjecture} of Bressan \cite{Bressan, B2} remains an intriguing open problem.  However, its generalized version, with any $\kappa>0$ and
\eqref{1.2} replaced by
\beq\lb{2.1}
\sup_{t>0} \|\nabla u(\cdot,t)\|_p \le 1 
\eeq
for some $p\ge 1$  (for $p=2$ this is the {\it bounded enstrophy} case), was proved for any $p>1$ by Crippa and De Lellis \cite{CL}.  Due to a relationship between mixing in the sense of Definition~\ref{D.1.1} and in terms of the $H^{-1}$-norm, discussed after Corollary \ref{C.1.5}, this can be extended to the same result for mixing in the latter sense \cite{IKX, S}.  

\noindent
{\bf Mixing for general functions.}  
Our first goal here is to study the complementary question of how efficient mixing by incompressible flows actually {\it can} be, that is, obtaining {\it upper bounds} on best possible mixing times by flows satisfying \eqref{2.1}.  We do so by constructing very efficient mixing flows for {\it general mean-zero functions} $\rho_0\in L^\infty (Q)$ and any $\kappa>0$, with any of the three types of boundary conditions.  Our main mixing results  have {\it $\rho_0$-independent bounds} and are as follows.

%

\begin{theorem} \lb{T.1.1}
Consider incompressible flows $u:Q \times \bbR^+ \to \mathbb{R}^2$ satisfying \eqref{2.1} for some $p\in[1,\infty]$ and the no-flow  boundary condition  on $\partial Q\times\bbR^+$.  For each $p\in[1,\infty]$, there is $C_p<\infty$ such that for any mean-zero $\rho_0 \in L^\infty(Q)$ the following holds.
\begin{enumerate}
\item If $p\in[1,\tfrac{3+\sqrt 5}2)$, then there is $u$ as above  which, for any $\kappa,\eps\in(0,\tfrac 12]$,  $\kappa$-mixes $\rho_0$ to scale $\eps$ in time $C_p|\log(\kappa\eps)|$.
\item  If $p=\frac{3+\sqrt{5}}{2}$ ($=\text{golden ratio}+1$), then for any $\kappa\in (0,\tfrac 12]$ there is $u$ as above which $\kappa$-mixes $\rho_0$ to any scale $\eps\in(0,\tfrac 12]$  in time   $C_p|\log(\kappa \eps)|\, |\log \frac{|\log(\kappa\eps)|}{\kappa} |^{1/p}$.
\item If $p\in(\frac{3+\sqrt{5}}{2},\infty]$ and $\nu_p:=\tfrac{p^2-3p+1}{3p^2-p}$ (so $\nu_p\le\tfrac 13$, and $\nu_\infty=\tfrac 13$), then for any $\kappa,\eps\in(0,\tfrac 12]$, there is $u$ as above which $\kappa$-mixes $\rho_0$ to scale $\eps$ in time $C_p\kappa^{-\nu_p} |\log(\kappa\eps)|^{1+\nu_p}$.
 The flow $u$ can be made independent of $\eps$ if we only require that it $\kappa$-mixes $\rho_0$ to scale $\eps$ in time $C_p\kappa^{-\nu_p} |\log(\kappa\eps)|^{1+\nu_p}\log|\log(\kappa\eps)|$. 
\end{enumerate}
\end{theorem}

\begin{theorem} \lb{T.1.2}
Theorem \ref{T.1.1} continues to hold when the no-flow boundary condition is replaced by the periodic boundary condition.
\end{theorem}

\begin{theorem} \lb{T.1.3}
Theorem \ref{T.1.1} continues to hold when the no-flow boundary condition is replaced by the no-slip boundary condition, with the following changes.  For $p<\infty$, the term $C_p\kappa^{-1+1/p}$ is added to each mixing time (so, in particular, $u$ in (1) also depends on $\kappa$).  For $p=\infty$, the mixing times are changed to $C_p\kappa^{-1} |\log(\kappa\eps)|^2$ for $\eps$-dependent flows and $C_p\kappa^{-1} |\log(\kappa\eps)|^2 (\log|\log(\kappa\eps)|)^2$ for $\eps$-independent flows.
\end{theorem}

{\it Remarks.}  1. Of particular interest is the $\eps$-dependence of the mixing times, for a fixed $\kappa>0$ \cite{Bressan, B2,CL}.  Due to the abovementioned result from \cite{CL}, our $O(|\log\eps|)$ upper bound on the shortest mixing time is exact for $p\in(1,\tfrac {3+\sqrt 5}2)$.  We do not know whether our bound for $p\ge \tfrac {3+\sqrt 5}2$ is optimal for general $\rho_0$, or whether the value $p=\tfrac{3+\sqrt 5}2$ is indeed critical here. 

2. Since the flow in Theorems \ref{T.1.1}(1) and \ref{T.1.2}(1) is independent of both $\eps$ and $\kappa$, taking $\kappa=\eps$ yields $\eps$-mixing in time $O(|\log\eps|)$ for each $\eps$.  In the special case of periodic boundary conditions,  initial value $\rho_0=\chi_{(0,1/2)\times(0,1)} - \chi_{(1/2,1)\times(0,1)}$,  and \eqref{2.1} replaced by $\sup_{t>0} \|u(\cdot, t)\|_{\dot{BV}} \le 1$, an $O(|\log\eps|)$ upper bound was previously obtained in \cite{B2,LLNMD}.

3. The jump in the power of $|\log(\kappa\eps)|$ at $p=\infty$ in Theorem \ref{T.1.3} is due to the no-slip condition having an exponentially decreasing  effect (at the rate $2^{-n/p}$) for $p<\infty$ as our flow acquires progressively smaller scales (of size $2^{-n}$, with $n\lesssim |\log(\kappa\eps)|$).  This is then controlled by other terms in the relevant estimates.  For $p=\infty$ this effect stays large at all scales and becomes the dominant term.  The details are in the proofs of Theorems \ref{T.5.4} and \ref{T.5.3}.

4.  It is easy to see that if $\rho_0$ is supported away from $\partial Q$, then the bounds from Theorem \ref{T.1.1}  also hold in Theorem \ref{T.1.3}, albeit with $C_p$ also depending on $\dist(\supp\, \rho_0,\partial Q)$.

\begin{proof}[Proofs]
Theorem \ref{T.1.1} is equivalent to Theorems \ref{T.3.4}, \ref{T.4.5}, and \ref{T.4.6};  Theorem \ref{T.1.2} is equivalent to Theorem \ref{T.5.1}; and Theorem \ref{T.1.3} is equivalent to Theorems \ref{T.5.2}, \ref{T.5.3}, and \ref{T.5.4}.  The equivalences are obtained by noticing that if $J(t):=\int_0^t \|\nabla u(\cdot, s)\|_p\, ds$, then the solution $\til\rho$ of \eqref{1.1} with the flow  $\til u(\cdot, J(t)):=\|\nabla u(\cdot, t)\|_p^{-1} u(\cdot,t)$ (and $\til\rho(\cdot,0)=\rho_0$) satisfies $\til \rho (\cdot, J(t)) = \rho(\cdot,t)$.
\end{proof}

{\it Remark.}  The flows in this paper will all be piece-wise constant (and hence discontinuous) in time (including the rescaled flow in the above proof).  However, continuity or smoothness in time is easily obtained by a change of the time variable on each interval  on which the flow is constant.  For instance, if $u(\cdot,t)=u_0(\cdot)$ for $t\in [t_0,t_1]$, we may take $\til u(\cdot, t):=\alpha(t)u_0(\cdot)$ for $t\in [t_0,t_1]$ and some $\alpha\in C_c^\infty([t_0,t_1])$ with $\int_{t_0}^{t_1} \alpha(t) dt=t_1-t_0$.  Indeed, if $\til\rho$ solves \eqref{1.1} with $\til u$ in place of $u$ and $\til\rho(\cdot,0)=\rho(\cdot,0)$, then $\til\rho(\cdot,t)=\rho(\cdot,t_0+\int_{t_0}^t \alpha(s)ds)$ for all $t\in[t_0,t_1]$.



A natural question is what happens if $\|\nabla u(\cdot,t)\|_p$ is replaced by $\| u(\cdot,t)\|_{\dot W^{s,p}}$ in \eqref{2.1}.  The flows we construct throughout this paper have a ``self-similar'' nature --- they are ``turbulent'' at an exponentially decreasing sequence of scales as time progresses --- that allows us to answer this question rather easily.   Let us consider the family of squares (cells) 
\[
Q_{nij}:=\left(\frac i{2^n}, \frac {i+1}{2^n}\right)\times \left(\frac j{2^n}, \frac {j+1}{2^n}\right),
\]
with $n\ge 0$ and $i,j \in \{0,\ldots, 2^n - 1\}$.  (Note that $\{Q_{nij}\}_{i,j=0}^{2^n-1}$ tile $Q$ for each fixed $n$.)  Let us also consider only the $\eps$-independent flows from Theorems \ref{T.1.1}--\ref{T.1.3}, before the rescaling from the above proof, that is, as in the proofs of the theorems mentioned there.  Those proofs show that for each such flow $u$ and each $n\ge 0$ we have that $u|_{(n,n+1]}$ keeps all the $Q_{nij}$ invariant (``self-similarity'')
and for some $C_p'<\infty$ we have either 
\[
\sup_{t\in (n,n+1]} \|\nabla u(\cdot,t)\|_p \le C_p' \qquad\text{and}\qquad \sup_{t\in (n,n+1]} \|u(\cdot,t)\|_\infty \le C_p' 2^{-n}
\]
 in Theorems \ref{T.1.1}(1) and \ref{T.1.2}(1), or or the same with extra factors $\kappa^{-1} n(\log n)^2$  on the right-hand sides 
in the other cases  (these latter worst case bounds are achieved for $p=\infty$ in Theorem \ref{T.1.3}).   The interpolation inequality $\|u\|_{\dot W^{s,p/s}} \leq C_{s,p} \|\nabla u\|_{p}^s \|u\|_{\infty}^{1-s}$
for $s\in[0,1]$ (see, e.g., \cite[p.~536]{M}) now yields either
\[
\sup_{t\in (n,n+1]} \| u(\cdot,t)\|_{\dot W^{s,p/s}} \le C_{s,p}' 2^{-(1-s)n}
\]
in Theorems \ref{T.1.1}(1) and \ref{T.1.2}(1), or the same with an extra factor $\kappa^{-1} n(\log n)^2$  on the right-hand side in the other cases. Then the rescaling in time from the above proof (but corresponding to the $\dot W^{s,p/s}$-norm) provides a flow that is uniformly bounded in time in $\dot W^{s,p/s}(Q)$ and {\it $\kappa$-mixes $\rho_0$ to scale 0 in finite time $\tau$} (in the sense that for each $\eps>0$ there is $\tau_\eps<\tau$ such that $\rho(\cdot,t)$ is $\kappa$-mixed to scale $\eps$ for each $t\in(\tau_\eps,\tau)$).  This yields the following. 

\begin{corollary} \lb{C.1.5}
Consider incompressible flows $u:Q \times (0,\tau) \to \mathbb{R}^2$ satisfying 
\[
\sup_{t\in (0,\tau)} \| u(\cdot,t)\|_{\dot W^{s,p}} \le 1
\]
for some $p\in[1,\infty]$ and $s\in[0,1)$, any one of the three  boundary conditions above on $\partial Q\times\bbR^+$, and also $\sup_{t\in (0,\tau-\delta)} \|\nabla  u(\cdot,t)\|_{\max\{sp,1\}}<\infty$ for each $\delta>0$.  
\begin{enumerate}
\item If $s<\tfrac{3+\sqrt 5}{2p}$ and the boundary condition is no-flow or periodic, then there is $\tau_{s,p}<\infty$ such that for any mean-zero $\rho_0 \in L^\infty(Q)$ there is
$u$ as above with $\tau:=\tau_{s,p}$ which, for any $\kappa\in(0,\tfrac 12]$,  $\kappa$-mixes $\rho_0$ to scale 0 in time $\tau_{s,p}$.  
\item  If either $s\ge \frac{3+\sqrt{5}}{2p}$ or the boundary condition is no-slip, then for any $\kappa\in (0,\tfrac 12]$ there is $\tau_{s,p,\kappa}<\infty$ such that for any mean-zero $\rho_0 \in L^\infty(Q)$ there is $u$ as above with $\tau:=\tau_{s,p,\kappa}$ which $\kappa$-mixes $\rho_0$ to scale 0 in time $\tau_{s,p,\kappa}$. 
\end{enumerate}
\end{corollary}

{\it Remarks.}  1.  One can use the above argument to also show algebraic-in-$\eps$ time of $\kappa$-mixing of $\rho_0$ for $s>1$, but one needs to adjust our flows appropriately near the boundaries of the $Q_{nij}$ to make them belong to $\dot W^{s,p}(Q)$ in this case.  We leave the details to the reader.

2.  The above suggests that 1 is indeed the critical order of derivatives of $u$ in \eqref{2.1}. 

3. It is not difficult to show that this result, and the fact that $u|_{(n,n+1]}$ before the above rescaling keeps all the $Q_{nij}$ invariant, show that $\rho$ weak-$*$ converges in $L^\infty(Q)$  to 0 in (1) and to a function with values in $[-\kappa,\kappa]$ in (2) as $t\to\tau$.   

As we mentioned in the remark after Definition \ref{D.1.1}, another measure of the mixing scale of a mean-zero $f\in L^\infty(Q)$ is $\|f\|_{H^{-1}} \|f\|_{\infty}^{-1}$.
One can easily check that $f$ being $\kappa$-mixed to scale $\eps$ implies that the {\it mix-norm}  of $f$ \cite{MMP}  is bounded above by $C\sqrt{\eps+\kappa^2} \|f\|_{\infty}$ for some $C<\infty$ \cite[(8)--(10)]{MMP}. Hence from the equivalence between the mix-norm and the $H^{-1/2}$ norm \cite[Corollary 2.1]{MMP} one immediately has $\|f\|_{H^{-s}} \leq C\sqrt{\eps+\kappa^2} \|f\|_{\infty}$ for   all $s\geq \tfrac 12$.

On the other hand, Lemma \ref{L.A.2} shows that $\|f\|_{H^{-1}} \leq c\kappa^{3/2}\eps^2 \|f\|_{\infty}$ (with some $c>0$) implies that $f$ is $\kappa$-mixed to scale $\eps$.  This, combined with the result from \cite{CL}, immediately shows that there is $c_p>0$ such that if $\rho_0=\chi_{(0,1/2)\times(0,1)} - \chi_{(1/2,1)\times(0,1)}$ and an incompressible flow $u$ satisfies \eqref{2.1} for some $p\in(1,\infty]$, then  $\|\rho(\cdot,\tau)\|_{H^{-1}} \leq \eps \|\rho_0\|_{\infty}$ implies $\tau\ge c_p |\log \eps|$ (this  is also proved in \cite{IKX,S}).  One may again ask whether this $O(|\log\eps|)$ bound is achievable.

Parts (1) of Theorems \ref{T.1.1}--\ref{T.1.3} (and $\|f\|_{H^{-1}} \leq C\sqrt{\eps} \|f\|_{\infty}$ above when $\kappa:=\eps$) show that there indeed is $C_p<\infty$ such that for any mean-zero $\rho_0\in L^\infty(Q)$ there is  an incompressible flow $u$ satisfying \eqref{2.1} which yields $\|\rho(\cdot,\tau)\|_{H^{-1}} \leq \eps \|\rho_0\|_{\infty}$ for some $\tau\le C_p |\log \eps|$, 
albeit only for $p\in[1,\tfrac{3+\sqrt 5}2)$ in the case of no-flow and periodic boundary conditions (which includes the uniformly bounded enstrophy case, \eqref{2.1} with $p=2$), and  for $p=1$ in the case of no-slip boundary conditions.  (This obviously also yields a corresponding extension of Corollary \ref{C.1.5}.).  The difference is that for these $p$ we  achieve ``perfect'' mixing, with  $\int_{Q_{nij}}\rho(x,y,t)dxdy=0$ for any $t\ge n$ and  $i,j \in \{0,\cdots,2^n-1\}$, so $\kappa$ plays a less prominent role in the obtained bounds.  We are not able to do this in the other cases and, ironically, it turns out that the enemy to this effort is the possibility of $\rho(\cdot,n)$ being very well mixed inside some $Q_{nij}$ (but not near its boundary).  Unfortunately, we cannot discard this possibility for general $\rho_0$. 

A few days before  we finished writing the present paper, 
Alberti, Crippa, and Mazzucato  \cite{ACM} announced  that in the case of periodic boundary conditions they are able to obtain the above $O(|\log\eps|)$ bound (i.e., $\|\rho(\cdot,\tau)\|_{H^{-1}} \leq \eps \|\rho_0\|_{\infty}$ with $\tau\le C_{p} |\log \eps|$) for   $\rho_0=\chi_{(0,1/2)\times(0,1)} - \chi_{(1/2,1)\times(0,1)}$ and any $p\in[1,\infty]$.  For $p\in[1,\infty)$ they can prove the same result with $\rho_0$ being the characteristic function of any $A\subseteq\bbT^2$ with a smooth boundary and finitely many connected components (minus a constant to make $\rho_0$ mean-zero), albeit with the constant $C_p$ also depending on $A$.   Their method has a more geometric flavor than ours, but is also centered around flows with a ``self-similar'' structure, and they are able to obtain a better control of $\rho(\cdot,n)$ inside the cells $Q_{nij}$ for these special initial data and the flows they construct.  A paper with the proofs of the results announced in \cite{ACM} will appear later.





\noindent
{\bf Un-mixing for general sets.}  
Our second goal, closely related to the first, concerns the question of efficient {\it un-mixing} by incompressible flows.  Here it is natural to consider (measurable) sets $A\subseteq Q$ and ask how efficiently can they be transported by incompressible flows close to their un-mixed states $\til A :=(0,|A|)\times(0,1)$.  Or, equivalently (after time-reversal), how efficiently can the rectangle $\til A$ be transported close to a desired set $A$ of the same measure, instead of just being mixed.  
Hence, this is a more delicate question than that of mixing, albeit restricted to initial data which are characteristic functions of sets.  We are not aware of previous work in this direction.  The somewhat related but quite different phenomenon of {\it coarsening} has been studied before (e.g., in \cite{BOS,OSS,S2}). 

Obviously, the time of un-mixing, given the constraint \eqref{2.1}, will depend on the scale  to which $A$ is mixed.  By this we mean the scale $\eps=2^{-n}$ such that most of the squares $Q_{nij}$ are each mostly contained in $A$ or in $Q\setminus A$. Since this scale is given, it makes little sense to ask whether the constructed flows can be $\eps$-independent.  We will therefore drop \eqref{2.1}, require the un-mixing to happen in time 1, and try to minimize $\sup_{t\in(0,1]}\|\nabla u(\cdot,t)\|_p$ instead.  This is an equivalent question, due to rescaling in time, and will allow our flows to be $p$-independent.

Our main un-mixing result, illustrated in Figure \ref{Unmix_figure}, is now as follows (with the no-slip boundary condition, so the other two hold as well).

\begin{theorem} \label{T.6.1}
There is $C>0$ such that for any measurable $A\subseteq Q$, $n\ge 0$, and $\kappa\in(0,\tfrac 12]$, 
the following holds.  If at most $2^{2n}\kappa$ of the squares $\{Q_{nij}\}_{i,j=0}^{2^n-1}$ satisfy $2^{2n}|A\cap Q_{nij}| \in(\kappa,1-\kappa)$, then there is an incompressible flow $u:Q \times (0,1) \to \mathbb{R}^2$ with $u = 0$ on $\partial Q\times(0,1)$ and 
\beq\lb{6.1}
\sup_{t\in(0,1)} \|\nabla u(\cdot,t)\|_p \le C \kappa^{-1+1/p}n^{2-1/p} \qquad\text{for each $p\in[1,\infty]$}
\eeq
such that if $\rho$ solves \eqref{1.1} and $\rho(\cdot,0)=\chi_A$, then the set $B$ for which $\rho(\cdot,1)=\chi_B$ satisfies 
\[
\left|B\cap \left[(0,|A|)\times(0,1)\right]\right| \ge 1-2\kappa.
\]
\end{theorem}

\begin{figure}[htbp]
\includegraphics[scale=1]{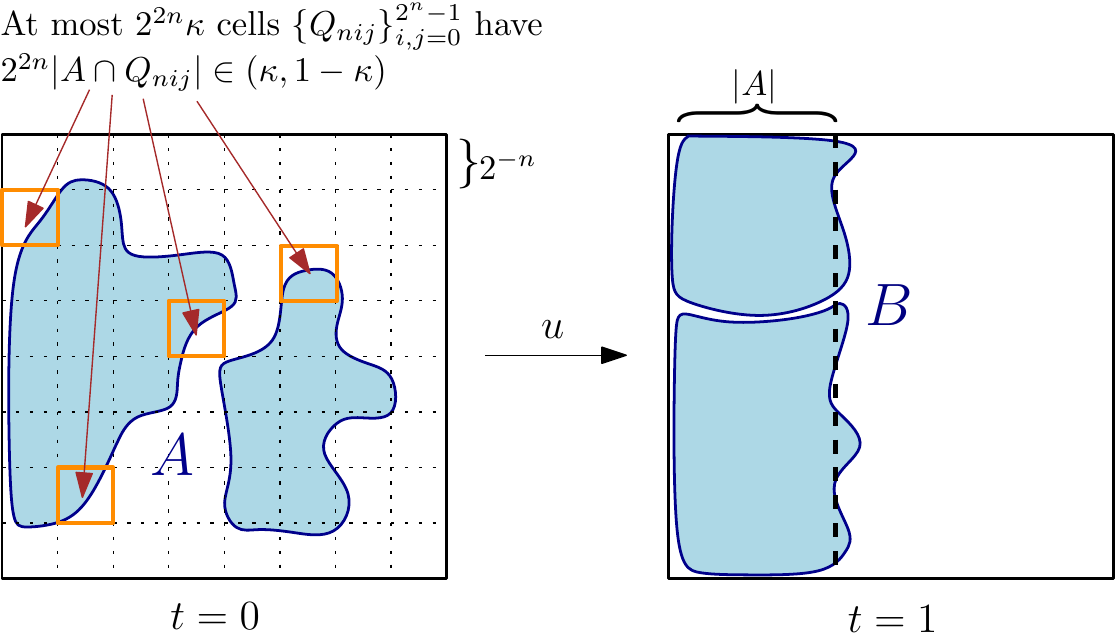}
\caption{An illustration of un-mixing from Theorem \ref{T.6.1}.}\label{Unmix_figure}
\end{figure}

{\it Remarks.}   1.  By $\rho$ solving \eqref{1.1} we mean that $\rho(\cdot,t):=\chi_{S(t)}$, with $S(t):= \{X_{(x,y)}(t)\,:\, (x,y)\in A\}$ and $X_{(x,y)}$ solving $X_{(x,y)}'(t)=u(X_{(x,y)}(t),t)$ and $X_{(x,y)}(0)=(x,y)$.

2. When \eqref{6.1} is replaced by $\sup_{t\in(0,1)} \| u(\cdot,t)\|_{\dot{BV}} \le Cn$, 
this also holds for $\kappa=0$ and the no-flow boundary condition $u\cdot n=0$ on $\partial Q\times(0,1)$ (see the end of the proof). 
 
 3.   Scaling of $u$ in time (which is different for different $p$) shows that if we require $\sup_{t\in(0,\tau)} \|\nabla u(\cdot,t)\|_p \le 1$, then the time $\tau$ of the above un-mixing satisfies $\tau\le C\kappa^{-1+1/p}n^{2-1/p}$.  That is, if $A$ is mixed to scale $\eps$ but not much more than that (in the sense of Theorem \ref{T.6.1}), it can be unmixed in time $O(|\log\eps|^{2-1/p})$.

4.  Similarly to the case of mixing, the ``self-similar'' structure of the flows we construct shows that Theorem \ref{T.6.1} holds with \eqref{6.1} replaced by $\sup_{t\in(0,1)} \|u(\cdot,t)\|_{\dot W^{s,p}} \le C_{s,p,\kappa}$ when $s\in[0,1)$ (i.e., the bound is independent of the scale $2^{-n}$).  Notice that for any $\kappa>0$, any measurable set $A\subseteq Q$ satisfies the hypotheses of the theorem for all large enough $n$.

Finally, here is an interesting corollary of our construction of un-mixing flows, related to the last remark.  It shows that for $s<\tfrac 1p$, a rectangle can be transformed into any measurable set of the same measure in finite time by an incompressible flow which is uniformly in time bounded in $\dot W^{s,p}(Q)$ and satisfies the no-flow boundary condition (it also has bounded variation, so that \eqref{1.1} is well-posed).  Notice that there are no errors and no $\kappa$ here.

\begin{corollary} \lb{C.6.3}
For any $s<\tfrac 1p$ there is $C_{s,p}<\infty$ such that for any measurable $A\subseteq Q$ there is an incompressible flow $u:Q \times (0,1) \to \mathbb{R}^2$ with $u\cdot n=0$ on $\partial Q\times(0,1)$, satisfying $\sup_{t\in(0,1-\delta)}\|u(\cdot,t)\|_{BV}<\infty$ for any $\delta>0$ and 
\beq \lb{6.8}
\sup_{t\in(0,1)} \|u(\cdot,t)\|_{\dot W^{s,p}} \le C_{s,p}, 
\eeq
such that if $\rho$ solves \eqref{1.1} and $\rho(\cdot,0)=\chi_{\{0<x<|A|\}}$, then $\lim_{t\to 1} \|\rho(\cdot,t)-\chi_A\|_1=0$.
\end{corollary}




\noindent
{\bf Organization of the paper.}  
  Theorem \ref{T.1.1}(1) is proved in Sections \ref{S2} and  \ref{S3}, and its parts (2) and (3) are proved in Section \ref{S4}.  Section \ref{S2} contains the simplest version of our method of construction of mixing flows, which only works for $p<2$.  The cases $p\in[2,\tfrac {3+\sqrt 5}2)$ and $p\in [\tfrac {3+\sqrt 5}2,\infty]$, treated in Sections \ref{S3} and \ref{S4}, are progressively more complicated.  However, in a remark at the beginning of Section \ref{S4} we provide for the convenience of the reader a relatively simple extension of the argument from Section \ref{S2}  which treats all $p\in[1,\infty]$ (as well as other boundary conditions), although the bounds obtained are worse than in Theorems \ref{T.1.1}--\ref{T.1.3}.  
  
 The proofs of Theorems \ref{T.1.2} and \ref{T.1.3} appear in Section \ref{S5}.  The un-mixing results are then proved in Section \ref{S6}, which only uses results from Section \ref{S2} (it is also closely related to the abovementioned remark in Section \ref{S4}).  Some technical lemmas are left for the Appendix.

\noindent
{\bf Acknowledgements.}  The authors wish to thank Inwon Kim, Alexander Kiselev, Andreas Seeger, Brian Street, and Jean-Luc Thiffeault for stimulating discussions.   YY acknowledges partial support by NSF grants DMS-1104415, DMS-1159133 and DMS-1411857. AZ acknowledges partial support by NSF grants DMS-1056327 and DMS-1159133.

\section{Perfect mixing for no-flow boundary conditions and $p<2$} \label{S2}

In this section, we will start with the simplest case, $p<2$ (plus the no-flow boundary condition $u \cdot n = 0$ on $\partial Q$), and show that the lower  bound $C_{p,\kappa} |\log\eps|$ on the  mixing time obtained by Crippa and De Lellis \cite{CL} is in fact attainable for these $p$. In the next section we will extend this result to all $p<\tfrac{3+\sqrt 5}2$.  We note that the flows we construct in this and the next section will in fact yield $\int_{Q_{nij}}\rho(x,y,t)dxdy=0$ for any $t\ge n$ and $i,j \in \{0,\cdots,2^n-1\}.$  This is what ``perfect mixing'' in the sections'  titles refers to.


We start with the construction of two stream functions $\psi$ and $\eta$, which will serve as the basic building blocks for the subsequent construction of our flow $u$.


\noindent
{\bf Construction of the stream functions.}
Let $\bar Q_c:=[0,1]^2\setminus \{ (0,0),(0,1),(1,0),(1,1)\}$ be the closed square without the corners and let $\partial Q_c:=\partial Q\cap \bar Q_c$ be its boundary without the corners.  For a {\it stream function} $\psi\in C(Q)$, denote $\nabla\psi:=(\psi_x,\psi_y)$ and $\nabla^2\psi:=(\psi_{xx},\psi_{xy},\psi_{yx},\psi_{yy})$.  If the level set $\{\psi=s\}$ is a simple closed curve, we define 
\[
T_\psi(s) := \int_{\{\psi=s\}} \frac 1{|\nabla \psi|} d\sigma.
\]
Notice that then $T_\psi(s)$ equals the time a particle advected by the (incompressible) flow 
\[
u=\nabla^\perp \psi := (-\psi_y,\psi_x)
\]
traverses the curve $\{\psi=s\}$.
If the level set is a point, we let $T_\psi(s):=\lim_{s'\to s} T_\psi(s')$, provided the (one-sided) limit exists.

\begin{proposition} \label{prop:stream}
There exists a stream function $\psi\in C(\bar Q)$, with $\nabla\psi$ continuous on $\bar Q_c$ and differentiable on $Q\setminus\{(\tfrac 12,\tfrac 12)\}$, such that:
\begin{enumerate}
\item $\psi>0$ on Q, $\psi=0$ on $\partial Q$, $\nabla \psi\in L^\infty(Q)$, and $\partial_n  \psi = -4$ on $\partial Q_c$;
\item $\nabla^2 \psi\in L^p(Q)$ for all $p\in[1,2)$;
\item  the level set $\{\psi=s\}$ for each $s\in[0,\|\psi\|_\infty)$ is a simple closed curve and for $s=\|\psi\|_\infty$ it is the point $(\tfrac 12,\tfrac 12)$, and $T_\psi(s) = 1$ for each $s\in[0,\|\psi\|_\infty]$. 
\end{enumerate}
\end{proposition}

We will obtain $\psi$ by modifying the stream function $\varphi$ from the following lemma, which satisfies (1) and (2), but not (3). The proof of the lemma is elementary but a little tedious, so we postpone it to the appendix.

\begin{lemma} \label{lemma:phi}
The function
\begin{equation}
\varphi(x,y) := \frac{4}{\pi} \,\frac{\sin(\pi x) \sin (\pi y)}{\sin(\pi x) + \sin(\pi y)}
\label{eq:def_phi}
\end{equation}
on $\bar Q$ (defined to be 0 in the four corners) satisfies:
\begin{enumerate}
\item $\varphi>0$ on Q, $\varphi=0$ on $\partial Q$, $\nabla \varphi\in L^\infty(Q)$ and $\partial_n  \varphi = -4$ on $\partial Q_c$;
\item $\nabla^2 \varphi\in L^p(Q)$ for all $p\in[1,2)$;
\item  the level set $\{\varphi=s\}$ for each $s\in[0,\tfrac 2\pi)$ is a simple closed curve and for $s=\tfrac 2\pi$ it is the point $(\tfrac 12,\tfrac 12)$, and $\sup_{s\in[0,2/\pi]} T_\varphi(s) <\infty$ and $T_\varphi(0)=1$;
\vspace{0.2cm}
\item $T_\varphi$ is differentiable on $(0,\tfrac 2\pi)$ and $\sup_{s\in(0,2/\pi)} |\log s|^{-1}|T_\varphi'(s)| \sup_{\{\varphi=s\}} |\nabla \varphi|^2 <\infty$.
\end{enumerate}
\end{lemma}

\begin{proof}[Proof of Proposition \ref{prop:stream}]
With  $\varphi$ from Lemma \ref{lemma:phi}, we let
\begin{equation}
\psi(x,y) := \int_0^{\varphi(x,y)} T_\varphi(s) ds,
\label{def:psi}
\end{equation}
so that $\varphi$ and $\psi$ share their level sets (although their values are different) and
\begin{equation}
\nabla \psi(x,y) = T_\varphi(\varphi(x,y)) \nabla\varphi (x,y).
\label{eq:dpsi}
\end{equation} 
The properties of $\nabla\psi$ and (1) now follow from the definition of $\varphi$ and Lemma \ref{lemma:phi}(1,3).  Since from \eqref{eq:dpsi} we have
\begin{equation}\label{2.4}
|\nabla^2 \psi | \leq |\nabla^2 \varphi| \, T_\varphi(\varphi(x,y)) + |\nabla \varphi|^2 \left|T'_\varphi(\varphi(x,y))\right|,
\end{equation} 
(2) holds due to (for some $C<\infty$)
\[
\begin{split}
\frac 1C\int_Q \left( |\nabla \varphi|^2 \left|T'_\varphi(\varphi)\right| \right)^p dxdy &\le \int_Q \left| \log \varphi \right|^2 dxdy \qquad\text{(by $p<2$ and Lemma \ref{lemma:phi}(4))} \\
&= \int_0^{2/\pi} |\log s|^2 \int_{\{\varphi=s\}}  \frac 1{|\nabla \varphi|} d\sigma ds \qquad\text{(by the co-area formula)}\\
&= \int_0^{2/\pi} |\log s|^2 T_{\varphi}(s)ds<\infty  \qquad\text{(by Lemma \ref{lemma:phi}(3))}. 
\end{split}
\]
Finally, if $F(s) := \int_0^s T_\varphi(s') ds'$ (and $F^{-1}$ is its inverse function), then Lemma \ref{lemma:phi}(3) and  
\[
T_{\psi}(s) 
=  \int_{\{\psi=s\}} \frac{1}{|\nabla \psi(s)|} ds =  \int_{\{\varphi = F^{-1}(s)\}}  \frac{1}{|\nabla\varphi| \, T_\varphi(F^{-1}(s))}d\sigma = 1
\]
yield (3).
\end{proof}

Let $\bar Q_{cc}:=\bar Q_c\setminus \{ (\tfrac 12,0),(\tfrac 12,1),(0,\tfrac 12),(1,\tfrac 12), (\tfrac 12,\tfrac 12)\}$ be $\bar Q$ without the 9 corners and centers of $Q$ and its sides, and let  $\partial Q_{cc}:=\partial Q\cap \bar Q_{cc}$.
Besides $\psi$ from Proposition \ref{prop:stream}, we will need the following stream function $\eta$ (which satisfies $\nabla \eta=\nabla\psi$ on $\partial Q_{cc}$).

\begin{proposition}  \label{prop2}
There exists a stream function $\eta\in C(\bar Q)$, with $\nabla\eta$ continuous on $\bar Q_{cc}$ and differentiable on $Q\setminus \left(\{x=\tfrac12\}\cup\{y=\tfrac12\}\cup\{x+y=\tfrac12\}\cup \{x+y=\tfrac32\} \right)$, such that:
\begin{enumerate}
\item $\eta=0$ on $\partial Q\cup\{x=\tfrac12\}\cup\{y=\tfrac12\}$, $\nabla \eta\in L^{\infty}(Q)$, and $\partial_n  \eta = -4$ on $\partial Q_{cc}$;
\item $\nabla^2 \eta\in L^p(Q)$ for all $p\in[1,2)$.
\end{enumerate}
\end{proposition}

\begin{proof}
Decompose $Q$ into two squares and four triangles, separated by the lines $\{x=\tfrac12\}$, $\{y=\tfrac12\}$, $\{x+y=\tfrac12\}$, $\{x+y=\tfrac32\}$ (see Figure \ref{domains}). On the two squares $Q_3$ and $Q_4$, we let $\eta(x,y)=\frac{1}{2}\psi(2(x-x_0),2(y-y_0))$, 
where $\psi$ is from Proposition \ref{prop:stream}, and $(x_0, y_0)$ is the lower left corner of $Q_3$ and $Q_4$, respectively. 

\begin{figure}[htbp]
\includegraphics[scale = 0.75]{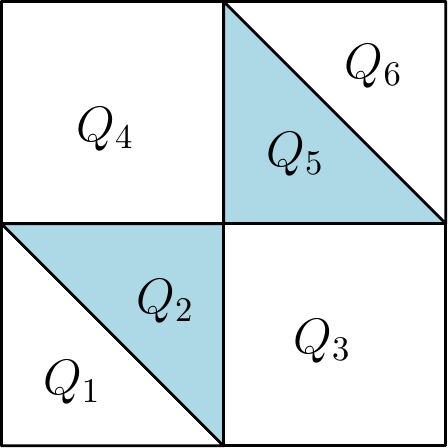}
\caption{The decomposition of $Q$ from the construction of $\eta$.}\label{domains}
\end{figure}

On $Q_1$ we let
\begin{equation}
\eta(x,y) = 4\left(\frac{1}{x} + \frac{1}{y} + \frac{\sqrt{2}}{\frac{1}{2}-x-y}\right)^{-1}.
\end{equation}
One can easily check that $\eta = 0$ on $\partial Q_1$ and $\partial_n \eta = -4$ on $\partial Q_1$ (except at the corners). Differentiation yields $\nabla \eta\in L^\infty(Q_1)$ and $d\nabla^2 \eta\in L^\infty(Q_1)$, where the function $d$ is the distance from the closest corner of $Q_1$.  It follows that $\nabla^2 \eta\in L^p(Q_1)$ for all $p\in[1,2)$.

Finally, in $Q_2$ we define $\eta$ by odd reflection across $\{x+y=\tfrac{1}{2}\}$ from $Q_1$ (in particular, $\nabla \eta$ is then continuous on $\partial Q_2$), and in $Q_5\cup Q_6$ by even reflection across $\{x+y=1\}$ from $Q_1\cup Q_2$.  The desired properties of $\eta$ on $Q$ then follow immediately from the above properties of $\eta$ on $Q_1$ and the properties of $\psi$. Note also that  $\eta>0$ on $Q_1\cup Q_3 \cup Q_4 \cup Q_6$ (white in Figure \ref{domains}) and $\eta<0$ on $Q_2\cup Q_5$ (blue in Figure \ref{domains}). 
\end{proof}



\noindent
{\bf Construction of the mixing flows.}  We are now ready to prove our first mixing result.

\begin{theorem} \label{thm_p<2}
For any mean-zero $\rho_0 \in L^\infty(Q)$, there is an incompressible $u:Q \times \bbR^+ \to \mathbb{R}^2$ with  $u\cdot n = 0$ on $\partial Q\times\bbR^+$ such that for any $\kappa,\eps\in(0,\tfrac 12]$, the flow $u$ $\kappa$-mixes $\rho_0$ to scale $\eps$ in a time $\tau_{\kappa,\eps}$ satisfying
\beq \lb{2.1a}
\int_0^{\tau_{\kappa,\eps}} \|\nabla u(\cdot,t)\|_p \,dt \le C_p |\log(\kappa\eps)|
\eeq
for each $p\in[1,2)$, with $C_p<\infty$ depending only on $p$.
\end{theorem}

\begin{proof}
%
We will construct a flow as above with $\sup_{t>0} \|\nabla u(\cdot,t)\|_p\le C_p'$ for each $p\in[1,2)$ (and $C_p'<\infty$ depending only on $p$), such that $\int_{Q_{nij}}\rho(x,y,n)dxdy=0$ at any integer time $n$ and any $i,j \in \{0,\cdots,2^n-1\}.$  The theorem then immediately follows by taking $\tau_{\kappa,\eps}:= \lceil  | \log_2 (\kappa\eps)| \rceil + 2$, with $C_p$ such that $C_p \log r\ge C_p'(\lceil  \log_2 r \rceil + 2)$ for all $r\ge 4$.  This is because  it is easily shown that for $n:=\tau_{\kappa,\eps}$ and any $(x,y)\in Q$, the squares $Q_{nij}$  which are fully contained in $B_\eps(x,y)\cap Q$ have total area $\ge (1-\kappa)|B_\eps(x)\cap Q|$.  Hence it remains to construct such a flow.

Obviously, $\int_{Q_{nij}}\rho(x,y,n)dxdy=0$ for $n=0$ and all $i,j$ (that is, $i,j=0$ when $n=0$) because $\rho_0$ is mean-zero.  We will now proceed inductively, assuming this property holds for some (fixed from now on) $n\ge 0$ and constructing the flow $u$ on the time interval $[n,n+1]$ so that it also holds for $n+1$.

For any square $Q_{nij}$
, and for all $t\in(n,t_{nij}]$ (with $t_{nij}\in[n,n+\tfrac 12]$ to be determined), let $u$ in $Q_{nij}$ be the ``cellular'' flow
\begin{equation}
u (x,y, t)=  \nabla^\perp \left[ (-1)^{i+j} 2^{-2n} \psi \big(2^n x-i, 2^n y-j)\big) \right].
\label{u1}
\end{equation}
Proposition \ref{prop:stream}(3) and symmetry tells us that this flow rotates each $Q_{nij}$  by $180\degree$ by time $n+\frac{1}{2}$. So if $Q_{nij}'$ and $Q_{nij}''$ are the left and right halves of $Q_{nij}$,  the hypothesis $\int_{Q_{nij}} \rho(x,y,n)dxdy=0$ and continuity of $\int_{Q_{nij}'} \rho(x,y,t)dxdy$ in $t$ show that for each $Q_{nij}$ there exists  $t_{nij} \in [n, n+\frac{1}{2}]$  such that $\int_{Q_{nij}'} \rho(x,y,t_{nij})dxdy=0$.  Of course, then also $\int_{Q_{nij}''} \rho(x,y,t_{nij})dxdy=0$.

\begin{figure}[htbp] 
\includegraphics{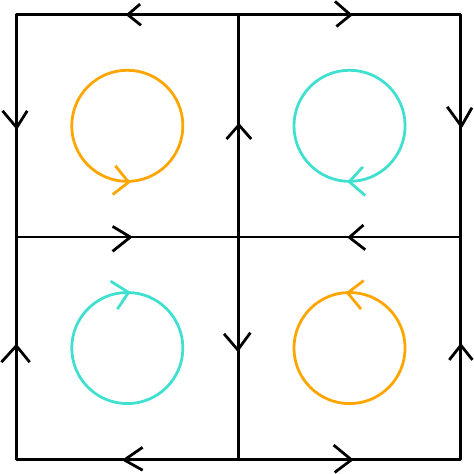} \hspace{2cm} \includegraphics{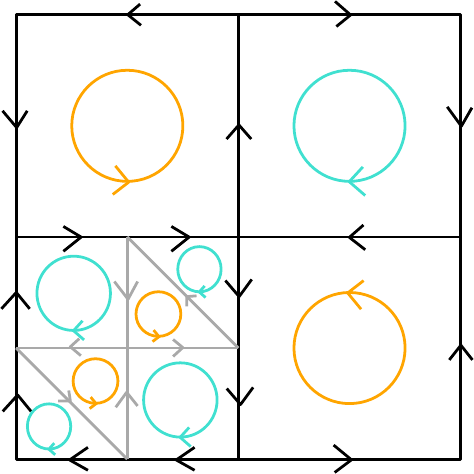} 
\caption{The mixing flow in four adjacent squares, as the cellular flow (left) is switched to the time-wasting flow (right) in one of them at $t_{nij}\in [n, n+\frac{1}{2}]$.}\label{flowpic}
\end{figure}

For $t\in(t_{nij},n+\tfrac 12]$, we let $u$ in $Q_{nij}$ be the ``time-wasting'' flow (see Figure \ref{flowpic})
\begin{equation}
u(x,y,t) =  \nabla^\perp \left[ (-1)^{i+j} \, 2^{-2n} \eta\big(2^n x - i, 2^n y - j\big) \right].
\label{u2}
\end{equation}
Proposition \ref{prop2}(1) shows that this flow does not cross $\partial Q_{nij}'$ and $\partial Q_{nij}''$, hence
\[
\int_{Q_{nij}'} \rho\left(x,y,n+\frac 12\right)dxdy=0=\int_{ Q_{nij}''} \rho\left(x,y,n+\frac 12\right)dxdy.
\]

The flow for $t\in (n+\tfrac 12, n+1]$ is constructed in the same fashion, but with the role of $Q_{nij}$ played by both $Q_{nij}'$ and $Q_{nij}''$.  That is, we decompose $Q$ into $2^{2n+1}$ identical rectangles $\til Q_{nij}:=(2^{-(n+1)}i, 2^{-(n+1)}(i+1))\times ( 2^{-n}j, 2^{-n}(j+1))$, 
so that by the above  $\int_{\til Q_{nij}} \rho\left(x,y,n+\frac 12\right)dxdy=0$ for each of them. In each $\til Q_{nij}$ we let
\[
u (x,y,t)= 
\begin{cases}
 \nabla^\perp \left[ (-1)^{i+j} 2^{-(2n+1)} \psi\big(2^{n+1} x-i , 2^n y-j\big) \right] & t\in (n+\tfrac 12,\til t_{nij}], 
\\  \nabla^\perp \left[ (-1)^{i+j} 2^{-(2n+1)} \eta\big(2^{n+1} x-i,  2^n y-j\big) \right] & t\in (\til t_{nij},n+1], 
\end{cases}
\]
where $\til t_{nij}\in[n+\tfrac 12, n+1]$ is such that $\int_{\til Q_{nij}'} \rho(x,y,\til t_{nij})dxdy=0$, with $\til Q_{nij}'$ the lower half of $\til Q_{nij}$.  It follows that $\int_{Q_{(n+1)ij}}\rho(x,y,n+1)dxdy=0$ for any $i,j\in\{0,\ldots, 2^{n+1}-1\}$ and the induction step is completed.

Finally,  $u$ is obviously incompressible and satisfies the no-flow condition on $\partial Q$.  Moreover, for $t\in(n,n+\tfrac 12]$, $u$ is continuous on all of $Q$ except of the corners and centers of the squares $Q_{nij}$ and the centers of their sides.  This is because Propositions \ref{prop:stream}(1) and \ref{prop2}(1), and  the factor $(-1)^{i+j}$, show that each couple of neighboring squares have the same velocity (of magnitude $4\cdot 2^{-n}$) on their common boundary (except at its center).  Since  $\|\nabla u(\cdot, t)\|_{L^p(Q_{nij})}^p$ is clearly either $2^{-2n}\|\nabla^2 \psi\|_{L^p(Q)}^p$ or $2^{-2n}\|\nabla^2 \eta\|_{L^p(Q)}^p$ for each $Q_{nij}$ and each $t\in(n,n+\tfrac 12]$, it follows that $\|\nabla u(\cdot, t)\|_p$ is between $\|\nabla^2\psi\|_p$ and $\|\nabla^2\eta\|_p$ for these $t$.  A similar argument applies to $t\in(n+\tfrac 12,n+1]$, with the speeds being $2^{1-n}$ and $2^{2-n}$ on the horizontal and vertical boundaries of the $\til Q_{nij}$, respectively.  Thus $\sup_{t>0} \|\nabla u(\cdot,t)\|_p\le C_p':= 2\max\{\|\nabla^2\psi\|_p, \|\nabla^2\eta\|_p\}$, and Propositions \ref{prop:stream}(2) and \ref{prop2}(2) yield $C_p'<\infty$ for $p<2$.
\end{proof}

%
%
%

\section{Perfect mixing for  no-flow boundary conditions and $p< \frac{3+\sqrt 5}2$} \lb{S3}

For $p\geq 2$ we can no longer directly use the stream functions $\psi,\eta$ from the last section since $\nabla^2 \psi,\nabla^2\eta \notin L^2(Q)$. This is because $|\nabla^2 \varphi(x,y)|$ (with $\varphi$ from Lemma \ref{lemma:phi}) is inversely proportional to the distance of $(x,y)$ to the nearest corner of $Q$, and $|\nabla^2 \eta|$ diverges in the same manner near each of the 9 points in Figure \ref{domains}.  We will therefore modify  $\varphi, \eta$ near their respective problematic points to circumvent this issue, and then adjust $\psi$ accordingly.  Notice that this means that we also need to modify $\varphi$ near the four centers of the sides of $Q$, to match a cellular and a time-wasting flow in two neighboring cells.

\begin{lemma} \label{phi_a}
Let $P$ be the set containing the four corners of $Q$ and the four centers of its sides, and let $d_P(x,y):= dist((x,y),P)$ for $(x,y)\in\bar Q$. Let $f\in C^\infty(\bbR^+_0)$ be a non-decreasing function with $f(s) = 5s$ for $s\in[0,\tfrac1{10}]$ and $f(s)= 1$ for $s\ge \tfrac 15$. Let
\begin{equation}\varphi_a(x,y) = \varphi(x,y) f(d_P(x,y))^a,
\label{def_phi_a}
\end{equation}
where $\varphi$ is from \eqref{eq:def_phi}. If $a\in (0,1]$, then $\varphi_a$ satisfies:

\begin{enumerate}
\item $\varphi_a>0$ on Q, $\varphi_a=0$ on $\partial Q$, $\nabla \varphi_a\in L^{\infty}(Q)$, and $\partial_n \varphi_a = -4f(d_P(x,y))^a$ on $\partial Q_{c}$;
\item $\nabla^2 \varphi_a \in L^p(Q)$ for  $a\in(0,1)$ and all $p\in[1, \tfrac 2{1-a})$, and $\nabla^2 \varphi_1 \in L^\infty(Q)$;
\item  the level set $\{\varphi_a=s\}$ for each $s\in[0,\tfrac 2\pi)$ is a simple closed curve and for $s=\tfrac 2\pi$ it is the point $(\tfrac 12,\tfrac 12)$, and for $a\in(0,1)$ we have
$\sup_{s\in[0,2/\pi]} T_{\varphi_a}(s) <\infty$;
\item $T_{\varphi_a}$ is differentiable on $(0,\tfrac 2\pi)$ and  $\sup_{s\in(0,2/\pi)} s^{2a/(a+1)} |T_{\varphi_a}'(s)| \sup_{\{\varphi_a=s\}} |\nabla \varphi_a|^2 <\infty$.
\end{enumerate}
\end{lemma}


Figure \ref{level} shows a comparison of the level sets of $\varphi$ and $\varphi_a$.  
\begin{figure}[htbp]
\includegraphics[scale=0.4]{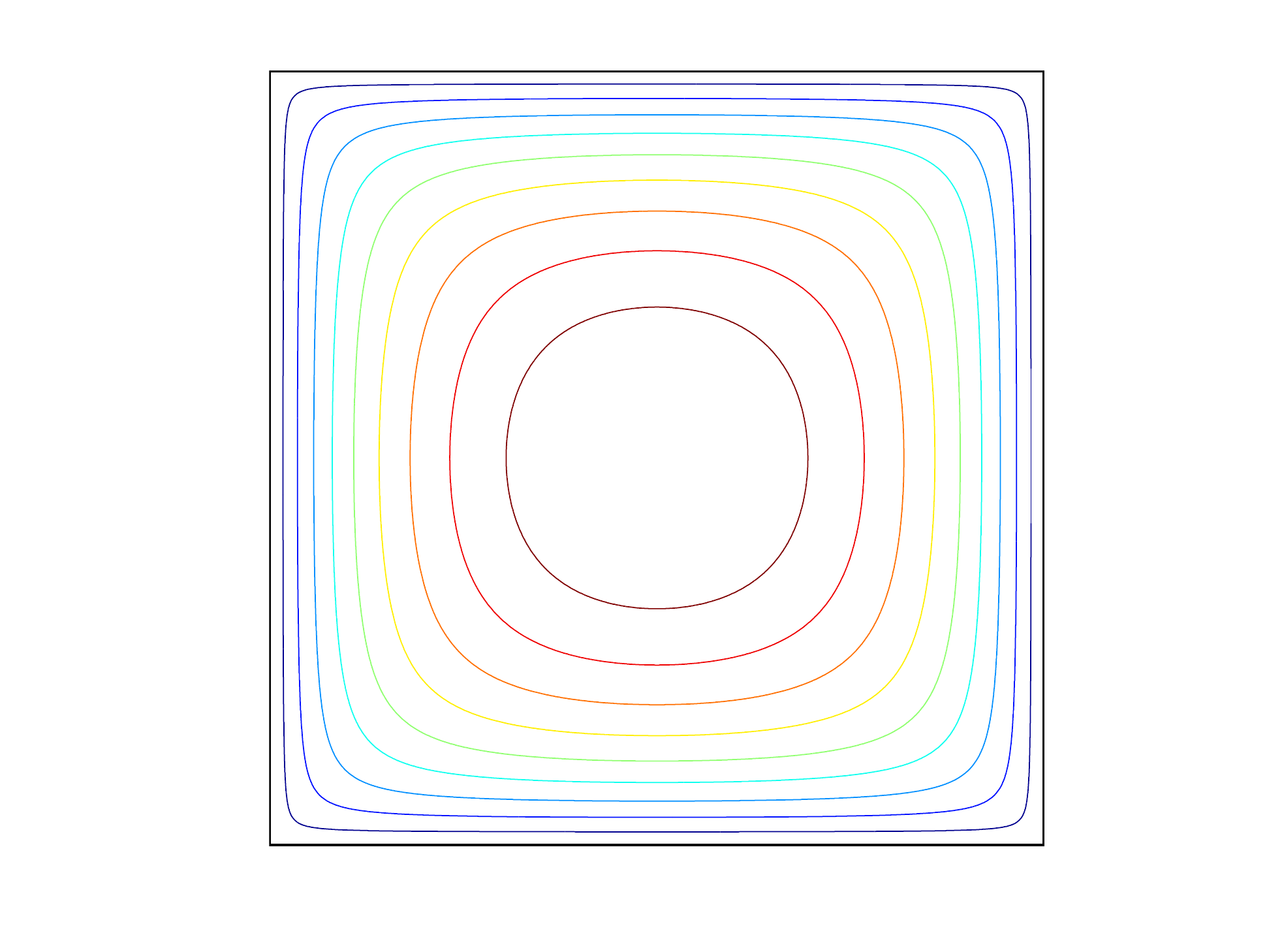}
\includegraphics[scale=0.4]{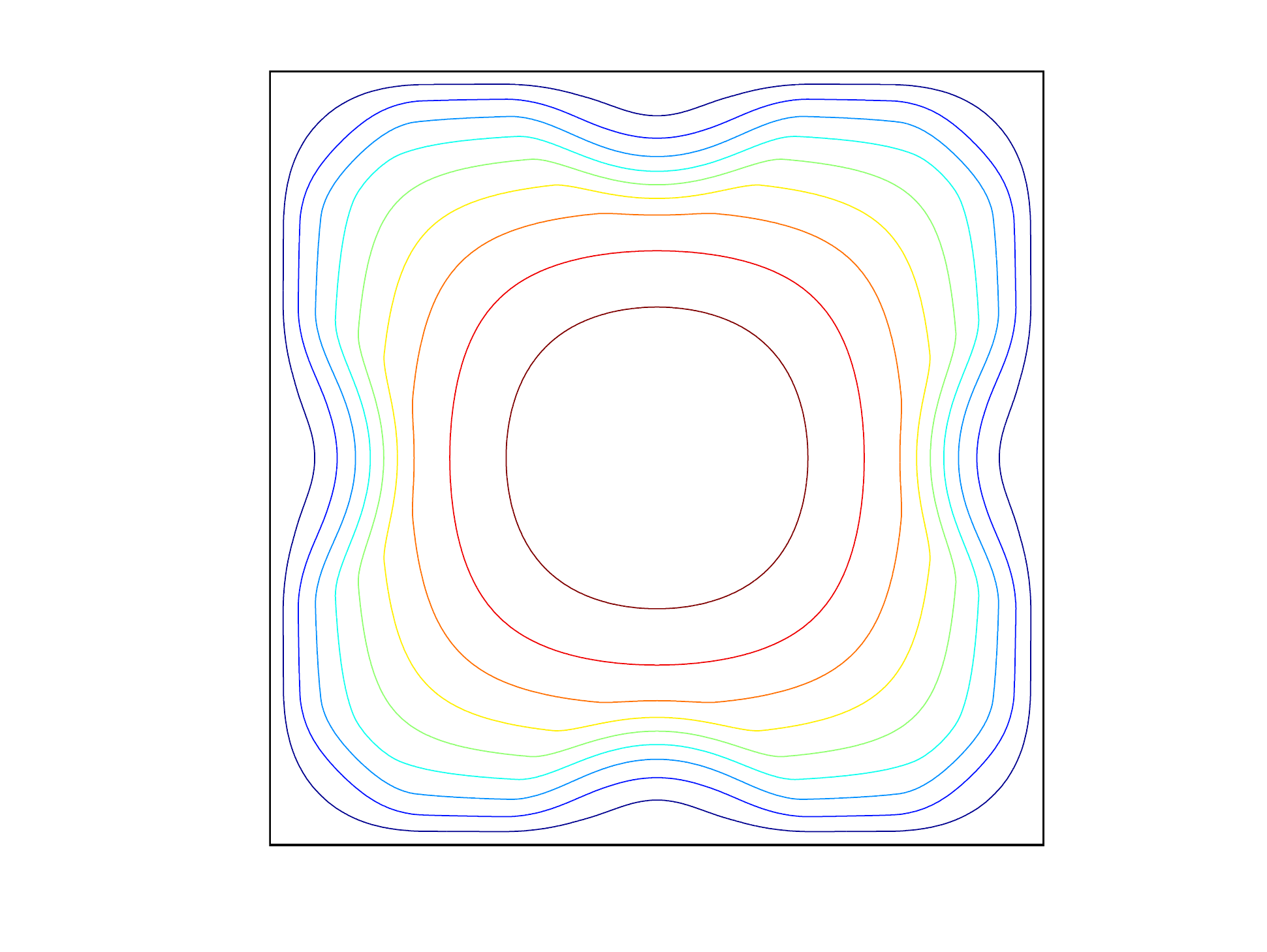}
\vspace{-0.8cm}
\caption{Level sets of $\varphi$ (left) and of $\varphi_a$ with $a=0.9$ (right).}\label{level}
\end{figure}
Just as with Lemma \ref{lemma:phi}, we postpone the proof of Lemma \ref{phi_a} to the appendix.  
Once we have $\varphi_a$, we can proceed as in Proposition \ref{prop:stream} and define a corresponding  stream function $\psi_a$ whose period $T_{\psi_a}(s)=1$ for each level set $\{\psi_a=s\}$.

\begin{proposition} \label{prop_a}
 For any $a\in(0,1)$ and  $\varphi_a$ from \eqref{def_phi_a}, let
\begin{equation}
\psi_a(x,y) :=\int_0^{\varphi_a(x,y)} T_{\varphi_a}(s) ds.
\label{psi_a}
\end{equation}
Then $\psi_a\in C(\bar Q)$, $\nabla\psi_a$ is continuous on $\bar Q_c$ and differentiable on $Q\setminus\{(\tfrac 12,\tfrac 12)\}$, and:
\begin{enumerate}
\item $\psi_a>0$ on Q, $\psi_a=0$ on $\partial Q$, $\nabla \psi_a\in L^{\infty}(Q)$, 
and $\partial_n \psi_a(x,y) = - 4T_{\varphi_a}(0) f(d_P(x,y))^a$ for $(x,y)\in \partial Q_c$;
\item $\nabla^2 \psi_a\in L^p(Q)$ for all $p\in[1, \min\{\frac{2}{1-a}, \frac{a+1}{2a}\})$; 
\item  the level set $\{\psi_a=s\}$ for each $s\in[0,\|\psi\|_\infty)$ is a simple closed curve and for $s=\|\psi_a\|_\infty$ it is the point $(\tfrac 12,\tfrac 12)$, and $T_{\psi_a}(s) = 1$ for each $s\in[0,\|\psi_a\|_\infty]$. 
\end{enumerate}
\end{proposition}

\begin{proof}
The properties of $\nabla \psi_a$ and (1) are immediate from Lemma \ref{phi_a}(1,3) and
\begin{equation}
\nabla \psi_a(x,y) =  \nabla\varphi_a (x,y) \,T_{\varphi_a}(\varphi_a(x,y)).
\label{eq:dpsi_a}
\end{equation} 
The proof of (3) is identical to that of Proposition \ref{prop:stream}(3).
Finally, differentiating \eqref{eq:dpsi_a} yields
\begin{equation} \lb{3.4}
|\nabla^2 \psi_a | \leq \underbrace{|\nabla^2 \varphi_a| \, T_{\varphi_a}(\varphi_a(x,y))}_{=: A(x,y)} + \underbrace{|\nabla \varphi_a|^2 |T'_{\varphi_a}(\varphi_a(x,y))|}_{ =: B(x,y)}.
\end{equation} 
By Lemma \ref{phi_a}(2,3) we have $A\in L^p(Q)$ for $p<\frac{2}{1-a}$.
Lemma \ref{phi_a}(1,3,4) and the co-area formula show for  any $p<\tfrac{a+1}{2a}$ and some $C<\infty$ (depending on $p,a$),
\[
\frac 1C\int_Q B^p dxdy \le \int_Q \varphi_a^{-\frac{2ap}{a+1}} dxdy = \int_0^{2/\pi} s^{-\frac{2ap}{a+1}}\int_{\{\varphi_a=s\}}  \frac 1{|\nabla \varphi_a|} d\sigma ds = \int_0^{2/\pi} s^{-\frac{2ap}{a+1}} T_{\varphi_a}(s)ds<\infty.
\]
Hence $B\in L^p(Q)$ for these $p$ and  (2) also follows.
\end{proof}

We next define a time-wasting flow $\eta_a$ with $\nabla \eta_a=\nabla \psi_a$ on $\partial Q_{cc}$.

\begin{proposition} \label{eta_a}
 For any $a\in(0,1)$, there exists a stream function $\eta_a\in C(\bar Q)$, with $\nabla\eta_a$ continuous on $\bar Q_{cc}$ and differentiable on $Q\setminus \left(\{x=\tfrac12\}\cup\{y=\tfrac12\}\cup\{x+y=\tfrac12\}\cup \{x+y=\tfrac32\} \right)$, such that:
\begin{enumerate}
\item $\eta_a=0$ on $\partial Q\cup\{x=\tfrac12\}\cup\{y=\tfrac12\}$, $\nabla \eta_a\in L^{\infty}(Q)$, and $\partial_n \eta_a (x,y) = -4 T_{\varphi_a}(0)  f(d_P(x,y))^a$ for $(x,y)\in \partial Q_{cc}$;
\item $\nabla^2 \eta_a \in L^p(Q)$ for all $p\in[1, \tfrac 2{1-a})$.
\end{enumerate}
\end{proposition}

\begin{proof}Let $\eta$ be from Proposition \ref{prop2}, 
$\tilde P := P \cup \{(\frac{1}{2}, \frac{1}{2})\}$ (with $P$ from Lemma \ref{phi_a}), and
 $$\eta_a(x,y) :=  T_{\varphi_a}(0)\eta(x,y) f(d_{\tilde P} (x,y))^a$$
 for $(x,y) \in Q$.  Then all the claims follow from the same properties for $\eta$ (with (2) proved as  Lemma \ref{phi_a}(2)).
\end{proof}

We can now repeat  the proof of Theorem \ref{thm_p<2}, this time
using the stream functions $\psi_a,\eta_a$ instead of $\psi,\eta$.  Proposition \ref{prop_a}(2) suggests to pick $a\in(0,1)$ which maximizes $\min\{\frac{2}{1-a}, \frac{a+1}{2a}\}$, that is, $a:=\sqrt{5}-2$. Then $\frac{2}{1-a}=\frac{a+1}{2a}=\tfrac{3+\sqrt{5}}{2}$, and we obtain the following  
improvement of Theorem \ref{thm_p<2}, with $\nabla u(\cdot,t)\in L^p(Q)$ for all  $p< \tfrac{3+\sqrt{5}}{2}$.

\begin{theorem} \lb{T.3.4}
Theorem \ref{thm_p<2} holds with  $p\in [1,2)$ replaced by $p\in [1,\tfrac{3+\sqrt{5}}{2})$.
\end{theorem}

{\it Remark.}  In fact, since we take $a=\sqrt{5}-2$ for all $p\in[1,\tfrac{3+\sqrt{5}}{2})$, the proof shows that our flow $u$ is independent of $p\in[1,\tfrac{3+\sqrt{5}}{2})$, in addition to being independent of $\kappa,\eps\in(0,\tfrac 12]$.



\section{Mixing for no-flow boundary conditions and $p\ge \frac{3+\sqrt{5}}{2}$} \lb{S4}

For $p\ge \frac{3+\sqrt{5}}{2}$, the construction from the previous section does not work because of the behavior of $\nabla^2\psi_a$ at $\partial Q$.  Indeed, the term $|\nabla^2\varphi_a|$ on the right hand side of \eqref{3.4} blows up as $d_P(x,y)^{a-1}$ (with $d_P$ from Lemma \ref{phi_a}) near the set $P$ by \eqref{d2phia}, while $|T'_{\varphi_a}(s)|\sim s^{-2a/(a+1)}$ at $s=0$ (i.e., near $\partial Q$) by Lemma \ref{phi_a}(4).  For both these to be in $L^p$, one needs $(1-a)p<2$ and $\tfrac {2a}{a+1}p<1$, but such $a$ exists only for $p< \frac{3+\sqrt{5}}{2}$.  

A solution to this problem is to ``give up'' on a small neighborhood of $\partial Q$, and not require the period of our stream functions to be 1 on the streamlines with $s\lesssim\delta$, for some  $\delta>0$.  (The affected region will have area $\sim \delta$.  We can then choose $\delta \sim \kappa |\log(\kappa\eps)|^{-1}$, so a flow analogous to that from the proof of Theorem \ref{thm_p<2} will still $\kappa$-mix $\rho_0$ to scale $\eps$ in time $\sim|\log(\kappa\eps)|$, although the bound on $\|\nabla u(\cdot,t)\|_p$ will now also depend on $\delta$.)

 {\it Remark.}  The easiest way of doing this is by replacing in the proof of Theorem \ref{thm_p<2} the stream functions $\psi$ and $\eta$ by $\psi^\delta(x,y):=\psi(x,y) f(\psi(x,y)/\delta)$ and 0, with $\delta>0$ small and $f$ from Lemma \ref{phi_a}.  Then $\nabla^\perp \psi^\delta=0$ on $\partial Q$ (which is why we can do not need a time-wasting flow), and properties of $\varphi,\psi$ show  that $\sup_{\delta>0} \delta^{1-1/p} \|\nabla^2 \psi^{\delta}\|_{p}<\infty$ for all $p\in[1,\infty]$ (see the proof of Lemma \ref{L.6.2}).  Notice that now we have $T_{\psi^\delta}(s)=1$ for $s\ge\delta$ because $\psi^\delta=\psi$ on  $D_{\delta} := \{\psi> \delta\}$ (with $|Q\setminus D_\delta|\le C\delta$). The times $t_{nij}\in[n,n+\tfrac 12]$ (and similarly $\til t_{nij}$) are now chosen so that $\int_{Q_{nij}' \cap D_{nij}} \rho(x,y,t_{nij}) dxdy = \int_{Q_{nij}'' \cap D_{nij}} \rho(x,y,t_{nij}) dxdy$, with $D_{nij}\subseteq Q_{nij}$ being the image of $D_{\delta}$ under the translation+dilation taking $Q$ to $Q_{nij}$ (so $|D_{nij}|\ge (1-C\delta)2^{-2n}$). This yields $|\int_{Q_{nij}'} \rho(x,y,n+\tfrac 12)dxdy-\tfrac 12\int_{Q_{nij}} \rho(x,y,n)dxdy|\le \tfrac{C\delta} 2 2^{-2n} \|\rho_0\|_\infty$, and eventually $|\fint_{Q_{nij}} \rho(x,y,n)dxdy | \leq 4Cn\delta \|\rho_0\|_\infty$ via induction on $n$ (the details of this argument are spelled out in the proof of Theorem \ref{T.4.3}).  Choosing again $n=\tau_{\kappa,\eps}\sim |\log(\kappa\eps)|$ and then $\delta\sim \tfrac \kappa n$ yields $\kappa$-mixing to scale $\eps$ in time $\tau_{\kappa,\eps}$ by a flow $u$ with $\sup_{t>0} \|\nabla u(\cdot,t)\|_p \le C_p (\tfrac  {|\log(\kappa\eps)|}\kappa )^{1-1/p}$.  It follows that Theorem \ref{thm_p<2} holds (for all $p\in[1,\infty]$ and any of our three boundary conditions)  with the right-hand side of \eqref{2.1a} being $C_p \kappa^{-1+1/p} |\log(\kappa\eps)|^{2-1/p}$.  We will now show how to improve this estimate for no-flow boundary conditions, and also make the power of $|\log(\kappa\eps)|$ converge to 1 as $p\downarrow \tfrac{3+\sqrt 5}2$, in two steps.  In Section \ref{S5} we treat the other boundary conditions.

For the sake of simplicity, let us start with the case $p=\infty$.

\begin{proposition}
\label{p41}
For any $a\in(0,1)$, $\delta\in (0,\tfrac1{10})$, there is a stream function $\psi_{a,\delta}\in C(\bar Q)$, with $\nabla\psi_{a,\delta}$ continuous on $\bar Q_c$ and differentiable on $Q\setminus\{(\tfrac 12,\tfrac 12)\}$, such that:
\begin{enumerate}
\item $\psi_{a,\delta}>0$ on Q, $\psi_{a,\delta}=0$ on $\partial Q$, $\nabla \psi_{a,\delta}\in L^\infty(Q)$, and $\partial_n \psi_{a,\delta}(x,y) = - N_{a,\delta}(d_P(x,y))$ for $(x,y)\in \partial Q_c$, for some function $N_{a,\delta}:[0,\tfrac 14]\to [0,\infty)$;
\item $\sup_{\delta\in(0,1/10)}  \delta^{\max\{1-a,2a\}/(a+1)} \|\nabla^2 \psi_{a,\delta}\|_{\infty} <\infty$ for each $a\in(0,1)$;
\item  there exists $s_{a,\delta}>0$, with $|\{\psi_{a,\delta}<s_{a,\delta}\}| \le\delta$
such that the level set $\{\psi_{a,\delta}=s\}$ for each $s\in[s_{a,\delta},\|\psi_{a,\delta}\|_\infty)$ is a simple closed curve and for $s=\|\psi_{a,\delta}\|_\infty$ it is the point $(\tfrac 12,\tfrac 12)$, and $T_{\psi_{a,\delta}}(s) = 1$ for each $s\in[s_{a,\delta},\|\psi_{a,\delta}\|_\infty]$.
\end{enumerate}
\end{proposition}

\begin{proof}
For any $a\in(0,1)$ and  $\delta\in(0,\frac{1}{10})$, let $D_{a,\delta} := \{(x,y)\in Q: \varphi_a(x,y)> \delta\}$, with $\varphi_a$ from Lemma \ref{phi_a}.  All constants below may depend on $a$ but not on $\delta$, unless specified.

The co-area formula and Lemma \ref{phi_a}(3) give for some $C<\infty$,
\beq\lb{4.3}
|Q\setminus D_{a,\delta}| = \int_0^{\delta} \int_{\{\varphi_a=s\}}  \frac 1{|\nabla \varphi_a|} d\sigma ds = \int_0^{\delta}  T_{\varphi_a}(s)ds \leq C\delta,
\eeq
and by \eqref{ineq_grad} we also have for some $c>0$ that $d_{a,\delta}:=c \delta^{1/(a+1)}$ satisfies
\begin{equation}
d_{a,\delta} \le \inf_{(x,y) \in D_{a,\delta/2}} d_P(x,y) \quad(<1).
\label{fact2}
\end{equation}
With $f$ from Lemma \ref{phi_a}, we now  let
\beq\lb{4.6}
\varphi_{a,\delta}(x,y) := \varphi_a(x,y)\, f \left( \frac{d_P(x,y)}{d_{a,\delta}} \right)^{1-a},
\eeq
so that $\varphi_{a,\delta} = \varphi_a$ when $d_P(x,y)\ge \tfrac 1{5}d_{a,\delta}$ (and in particular, on $D_{a,\delta/2}$).
Since 
\[
\varphi_{a,\delta}(x,y) =5d_{a,\delta}^{a-1} \varphi(x,y) d_P(x,y) = 5d_{a,\delta}^{a-1} \varphi_1(x,y)
\]
 when $d_P(x,y)\le \tfrac 1{10}d_{a,\delta}$,  
from \eqref{d2phia} for $a=1$ we obtain $|\nabla^2 \varphi_{a,\delta}(x,y)| \leq C d_{a,\delta}^{a-1}$ (for some $C<\infty$) when $d_P(x,y)\le \frac{1}{10}d_{a,\delta}$.  The same bound holds when $d_P(x,y)\in (\frac{1}{10}d_{a,\delta},\frac{1}{5}d_{a,\delta})$, due to \eqref{def_phi_a}, \eqref{ineq_grad}, and \eqref{d2phia}. By this and  \eqref{d2phia} for $a$, it follows for some $C<\infty$ (that changes between inequalities) and  all $(x,y)\in Q$,
\begin{equation}
|\nabla^2 \varphi_{a,\delta}(x,y)| \leq C\min\{ d_P(x,y)^{a-1}, d_{a,\delta}^{a-1} \} \le C \min\{d_P(x,y)^{a-1},\delta^{(a-1)/(a+1)}\} .
\label{out}
\end{equation}
Note also that $\partial_n \varphi_{a,\delta} (x,y) = -4f(d_P(x,y))^a f ( d_P(x,y)/d_{a,\delta})^{1-a}$ on $\partial Q_{c}$.  For later use we also mention that \eqref{4.6}, Lemma \ref{phi_a}(1), and $\varphi(x,y)d_{a,\delta}^{-1}\le \tfrac 15 \|\nabla\varphi\|_\infty$ for $d_P(x,y)< \tfrac 1{5}d_{a,\delta}$ yield 
\beq\lb{4.21}
\sup_{\delta\in(0,1/10)} \|\nabla\varphi_{a,\delta}\|_\infty <\infty.
\eeq

We now construct a new stream function $\psi_{a,\delta}$ by making the periods of all streamlines of $\varphi_{a,\delta}$ contained in $D_{a,\delta}$ (where $\varphi_{a,\delta}=\varphi_{a}$) to be 1. 
We let 
\[
T_{a,\delta}(s):=
\begin{cases}
T_{\varphi_a}(\delta)  & s\in[0,\tfrac \delta 2],
\\ T_{\varphi_a}(s) & s\in[\delta,\tfrac 2\pi],
\end{cases}
\]
and choose $T_{a,\delta}$ on $(\tfrac \delta 2,\delta)$ so that $T_{a,\delta}$ is differentiable on $(0,\tfrac 2\pi)$,
\beq\lb{4.4}
\sup_{s\le \delta} |T_{a,\delta}(s) - T_{a,\delta}(\delta)|\le \frac {T_{a,\delta}(\delta)}2 \qquad\text{and}\qquad 
\sup_{s\le \delta} |T_{a,\delta}'(s)| \le  |T_{\varphi_a}'(\delta)|.
\eeq
(We could have instead chosen $T_{a,\delta}(s) = T_{\varphi_a}(\delta)$ for $s\in[0, \delta)$, at the expense of $\nabla \psi_{a,\delta}$ not being differentiable on the streamline $\{\varphi_{a}=\delta\}$. This would not change our main results.)
We now define
\[
\psi_{a,\delta}(x,y) :=  \int_0^{\varphi_{a,\delta}(x,y)} T_{a,\delta}(s) ds.
\]
The properties of $\nabla \psi_{a,\delta}$ and (1) immediately follow from the properties of $\varphi_a$ and $f$, with $N_{a,\delta} (r):= 4f(r)^a f ( \tfrac r{d_{a,\delta}})^{1-a}T_{\varphi_a}(\delta)$. Part  (3) holds with $s_{a,\delta}:=\int_0^{\delta} T_{a,\delta}(s) ds$ and the estimate $|\{\psi_{a,\delta}<s_{a,\delta}\}| \le C\delta$ (which is sufficient because then one only needs to replace $\psi_{a,\delta}$ by $\psi_{a,\delta/C}$),  due to \eqref{4.3} and because $\psi_{a,\delta}-\psi_a$ is constant on $D_{a,\delta}= \{\psi_{a,\delta}>s_{a,\delta}\}$ (since $\varphi_{a,\delta}=\varphi_a$ there).


 To show (2), notice that  on $D_{a,\delta/2}$ we have by \eqref{d2phia}, Lemma \ref{phi_a}(3,4), \eqref{4.4}, and \eqref{fact2},
\[
\begin{split}
|\nabla^2 \psi_{a,\delta}| &\leq |\nabla^2 \varphi_a| \, T_{a,\delta}(\varphi_a(x,y)) + |\nabla \varphi_a|^2 |T_{a,\delta}'(\varphi_a(x,y))|\\
&\leq C d_P(x,y)^{a-1} + C\varphi_a(x,y)^{-2a/(a+1)}\\
&\leq C   \delta^{-\max\{1-a,2a\}/(a+1)},
\end{split}
\]
where $C<\infty$ changes between inequalities.
On $Q\setminus D_{a,\delta/2}$ we have $\psi_{a,\delta}=T_{a,\delta}(\delta)\varphi_{a,\delta}$, so  \eqref{out} and Lemma \ref{phi_a}(3) yield $|\nabla^2 \psi_{a,\delta}| \leq C\delta^{(a-1)/(a+1)}$ there.  These two estimates prove (2).
\end{proof}

We also define the  time-wasting flow corresponding to $\psi_{a,\delta}$.

\begin{proposition} \label{p42}
For any $a\in(0,1)$, $\delta\in (0,\tfrac1{10})$, there is a stream function $\eta_{a,\delta}\in C(\bar Q)$, with $\nabla\eta_{a,\delta}$ continuous on $\bar Q_{cc}$ and differentiable on $Q\setminus \left(\{x=\tfrac12\}\cup\{y=\tfrac12\}\cup\{x+y=\tfrac12\}\cup \{x+y=\tfrac32\} \right)$, such that:
\begin{enumerate}
\item $\eta_{a,\delta}=0$ on $\partial Q\cup\{x=\tfrac12\}\cup\{y=\tfrac12\}$, $\nabla \eta_{a,\delta} \in L^{\infty}(Q)$, and 
$\partial_n \eta_{a,\delta}(x,y) = - N_{a,\delta}(d_P(x,y))$ for $(x,y)\in \partial Q_{cc}$, with the function $N_{a,\delta}$ from Proposition \ref{p41};
\item $\sup_{\delta\in(0,1/10)}  \delta^{\max\{1-a,2a\}/(a+1)} \|\nabla^2 \eta_{a,\delta}\|_{\infty} <\infty$.
\end{enumerate}
\end{proposition}

\begin{proof}
Let
\[
\eta_{a,\delta}(x,y) := \eta_{a}(x,y)  f\left(\frac{d_P(x,y)}{d_{a,\delta}}\right)^{1-a} T_{\varphi_{a}} \left(\delta\right),
\]
where $d_{a,\delta}$ is from the previous proof. Then (1) follows from Proposition \ref{eta_a}(1) and the definition of $N_{a,\delta}$, and (2) is proved as Proposition \ref{p41}(2).
\end{proof}

Next, let us first obtain a weaker result for  $p=\infty$, with a $\sim |\log(\kappa\eps)|^{3/2}$ bound.
Afterwards, we will include an additional element to improve the bound to $\sim |\log(\kappa\eps)|^{4/3}=|\log(\kappa\eps)|^{1+\nu_\infty}$.

\begin{theorem}
\label{T.4.3}
For any mean-zero $\rho_0 \in L^\infty(Q)$ and any  $\kappa,\eps\in(0,\tfrac 12]$, there is an incompressible flow $u:Q \times \bbR^+ \to \mathbb{R}^2$ with  $u\cdot n = 0$ on $\partial Q\times\bbR^+$
 which $\kappa$-mixes $\rho_0$ to scale $\eps$ in a time $\tau_{\kappa,\eps}$ satisfying
\beq \lb{4.1a}
\int_0^{\tau_{\kappa,\eps}} \|\nabla u(\cdot,t)\|_\infty \,dt \le C \kappa^{-1/2} |\log(\kappa\eps)|^{3/2},
\eeq
 with a universal $C<\infty$.
The flow can be made independent of $\eps$ if the right-hand side of \eqref{4.1a} is replaced by $C\kappa^{-1/2} |\log(\kappa\eps)|^{3/2}\log|\log(\kappa\eps)|$. 
 \end{theorem}

\begin{proof}
Let $a:=\tfrac 13$ (which minimizes the power in Proposition \ref{p41}(2), to $\tfrac 12$), fix some $\delta\in(0,\tfrac 1{10})$ (to be chosen later), and let $\psi_{a,\delta},\eta_{a,\delta}$ be the corresponding stream functions from Propositions \ref{p41} and \ref{p42}.  
The construction of $u$ is now almost identical to the proof of Theorem \ref{thm_p<2}, with $\psi,\eta$ replaced by $\psi_{a,\delta},\eta_{a,\delta}$.  The one change is that Proposition \ref{p41}(3) only guarantees for each $n$ and any square $Q_{nij}$ (with $Q_{nij}', Q_{nij}''$ its left and right halves) existence of $t_{nij}\in[n,n+\tfrac 12]$ such that $\int_{Q_{nij}' \cap D_{nij}} \rho(x,y,t_{nij}) dxdy = \int_{Q_{nij}'' \cap D_{nij}} \rho(x,y,t_{nij}) dxdy$, with the set $D_{nij}\subseteq Q_{nij}$ such that $|D_{nij}|\ge (1-\delta)2^{-2n}$ (in fact, $D_{nij}$ is the image of the set $D_{a,\delta}=\{\psi_{a,\delta}> s_{a,\delta}\}$ under the translation+dilation taking $Q$ to $Q_{nij}$).  We thus find that 
$|\int_{Q_{nij}'} \rho(x,y,n+\tfrac 12)dxdy-\tfrac 12\int_{Q_{nij}} \rho(x,y,n)dxdy|\le \tfrac\delta 2 2^{-2n} \|\rho_0\|_\infty$.

A similar adjustment is made when finding the time $\til t_{nij}$ as in the proof of Theorem \ref{thm_p<2}.  We thus find that for any of the four squares with side length $2^{-(n+1)}$ which form $Q_{nij}$ (call it $\til Q$) we have
\beq\lb{4.23}
\left| \fint_{\til Q} \rho(x,y,n+1)dxdy-\fint_{Q_{nij}} \rho(x,y,n)dxdy \right|\le 4\delta  \|\rho_0\|_\infty.
\eeq
  Since $\rho_0$ is mean-zero, it follows by induction on $n$ that 
  \[
  \left|\fint_{Q_{nij}} \rho(x,y,n)dxdy \right| \leq 4n\delta \|\rho_0\|_\infty.
  \]

We now construct this flow on the time interval $[0,n]$, with $n:=\tau_{\kappa,\eps}:=\lceil | \log_2 \tfrac{\kappa\eps} 2| \rceil + 2$ (then it is easily shown that for any $(x,y)\in Q$, the squares $Q_{nij}$ which are fully contained in $B_\eps(x,y)\cap Q$ have total area $\ge (1-\tfrac\kappa 2)|B_\eps(x)\cap Q|$) and with $\delta := \frac{\kappa}{8n}$, which then obviously $\kappa$-mixes $\rho_0$ to scale $\eps$ in time $\tau_{\kappa,\eps}$.
As in the proof of Theorem \ref{thm_p<2}, but using Propositions \ref{p41}(2) and \ref{p42}(2), it follows that $\sup_{t\in(0,n]} \|\nabla u(\cdot,t)\|_\infty \le C' \delta^{-1/2}$ for some universal $C'<\infty$.  This yields the first claim because $C' \delta^{-1/2} \tau_{\kappa,\eps} \le C\kappa^{-1/2}|\log(\kappa\eps)|^{3/2}$ for some universal $C$.  

If we instead want the flow to be independent of $\eps$, we use on each time interval $[n,n+1]$ the flows $\psi_{1/3,\delta_n}, \eta_{1/3,\delta_n}$, with some $\delta_n>0$ to be chosen.  We then obtain 
\beq\lb{4.12}
\left|\fint_{Q_{nij}} \rho(x,y,n)dxdy \right| \leq 4 \|\rho_0\|_\infty \sum_{k=0}^{n-1}\delta_k,
\eeq
and $\sup_{t\in(n,n+1]} \|\nabla u(\cdot,t)\|_\infty \le C' \delta_n^{-1/2}$.  We again choose $n:=\tau_{\kappa,\eps}:=\lceil | \log_2 \tfrac{\kappa\eps} 2| \rceil + 2$, and then $\delta_k$ such that $\sum_{k=0}^\infty \delta_k\le \tfrac\kappa 8$, so that the obtained flow again $\kappa$-mixes $\rho_0$ to any scale $\eps\in(0,\tfrac 12]$ in time $\tau_{\kappa,\eps}$.  We now make the specific choice $\delta_{k-2}:=\tfrac{\kappa}M k^{-1}(\log k)^{-2}$ for $k\ge 2$, with $M:=8\sum_{k=2}^\infty k^{-1}(\log k)^{-2}$.  Then the integral in \eqref{4.1a} is bounded by $C' \sum_{k=0}^{n-1} \delta_k^{-1/2}$, which is no more than $C\kappa^{-1/2} |\log(\kappa\eps)|^{3/2}\log|\log(\kappa\eps)|$, for some large enough universal $C$.
\end{proof}

To achieve better mixing for $p=\infty$, we will next squeeze some mileage out of the sets $Q_{nij}\setminus D_{nij}$ from the previous proof, instead of simply ``giving up'' on them.  
To do this, we first need to obtain an estimate on the periods of the streamlines $\{\psi_{a,\delta}=s\}$ for $s\le\delta$.
This will say that even though these periods need not equal 1, most of them are still close to 1.  The following property of $\varphi_{a,\delta}$, which we again prove in the appendix, will yield the estimate.

\begin{lemma} \label{L.4.4}
For any $a\in(0,1)$, $\delta\in (0,\tfrac1{10})$, the functions $\varphi_{a,\delta}$ from \eqref{4.6} satisfy:
\begin{enumerate}
\item  the level set $\{\varphi_{a,\delta}=s\}$ for each $s\in[0,\tfrac 2\pi)$ is a simple closed curve;
\item 
$\sup_{0<s\le\delta<1/10}  (\delta^{(1-a)/(a+1)} \log \tfrac{2\delta} s )^{-1} | T_{\varphi_{a,\delta}}(s)- T_{\varphi_{a}}(\delta)|<\infty$ for each $a\in(0,1)$.
\end{enumerate}
\end{lemma}

We can now prove the final version of our result for $p=\infty$.

\begin{theorem} \label{T.4.5}
Theorem \ref{T.4.3} holds with each $\kappa^{-1/2} |\log(\kappa\eps)|^{3/2}$ replaced by $\kappa^{-1/3} |\log(\kappa\eps)|^{4/3}$.
 \end{theorem}


\begin{proof}
We proceed identically to the proof of Theorem \ref{T.4.3}, but with a slightly different choice of the times $t_{nij}\in[n,n+\tfrac 12]$ (and also $\til t_{nij}\in[n+\tfrac 12,n+1]$).  Let us define
\[
\til \psi_{a,\delta}(x,y) :=  \int_0^{\varphi_{a,\delta}(x,y)} T_{\varphi_{a,\delta}}(s) ds,
\]
so that  $T_{\til \psi_{a,\delta}}\equiv 1$
and $\til \psi_{a,\delta} - \psi_{a,\delta}$ is constant on $D_{a,\delta}=\{\varphi_{a,\delta}>\delta\}$.  The latter means that if we let $u$ and $\til u$ be given on $Q_{nij}\times(n,t_{nij}]$ by \eqref{u1} with $\psi_{a,\delta}$ and $\til \psi_{a,\delta}$ in place of $\psi$, respectively, then $u\equiv \til u$ on $D_{a,\delta}$.  Notice that $u$ is the same as in Theorem \ref{T.4.3}.

Let us now pick, in the proof of Theorem \ref{T.4.3}, the time $t_{nij}$ so that {\it if the flow  in $Q_{nij}$ for $t\in(n,t_{nij}]$ were $\til u$, then we would have $\int_{Q_{nij}'} \rho(x,y,t_{nij}) dxdy = \int_{Q_{nij}''} \rho(x,y,t_{nij}) dxdy$.}  This is possible because $\til u$ rotates $Q_{nij}$ by $180^\circ$ in time $\tfrac 12$.  Having this new $t_{nij}$, we still use $u$  to transport $\rho$ for $t\in(n,t_{nij}]$ because $\nabla \til u=\nabla^2 \til  \psi_{a,\delta}\notin L^\infty(Q)$.  This will introduce an error in the above equality of integrals of $\rho(\cdot,t_{nij})$ over $Q_{nij}'$ and $Q_{nij}''$, which we estimate by using Lemma \ref{L.4.4}.  After doing the same with $\til t_{nij}\in[n+\tfrac 12,n+1]$, we eventually still obtain an estimate like \eqref{4.23}, but with a better bound (see \eqref{4.24} below).  This is because Lemma \ref{L.4.4}(2) shows that most streamlines of $\psi_{a,\delta}$ lying in $Q\setminus D_{a,\delta}$ still have their periods $T_{\psi_{a,\delta}}$ close to 1.

  For the sake of simplicity, assume $n=0$ (so $Q_{nij}=Q$, $u=\nabla^\perp \psi_{a,\delta}$, $\til u=\nabla^\perp \til \psi_{a,\delta}$, and $t_{nij}=t_{000}\in[0,\tfrac 12]$), since the general case is identical.  We have $u\equiv \til u$ on $D_{a,\delta}$ and, in fact,
\[
\til u(x,y)=    \frac{ T_{\varphi_{a,\delta}}(\varphi_{a,\delta}(x,y))} {T_{a,\delta}(\varphi_{a,\delta}(x,y))} u(x,y)
\]
on $Q$. This and the definition of $t_{000}$ above (which uses $\til u$ instead of $u$) mean that if $X(0;x,y)=(x,y)$ and $X_t(t;x,y)=U(X(t;x,y))$, with 
\beq\lb{4.20}
U(x,y):= t_{000} (\til u(x,y)- u(x,y)) = t_{000} \frac{ T_{\varphi_{a,\delta}}(\varphi_{a,\delta}(x,y)) - T_{a,\delta}(\varphi_{a,\delta}(x,y))} { T_{a,\delta}(\varphi_{a,\delta}(x,y))}   u(x,y),
\eeq
then $\rho$  solving \eqref{1.1} with $u$ (not $\til u$) satisfies
\[
\int_{Q'} \rho \left(X(1;x,y),t_{000} \right) dxdy = \int_{Q''} \rho \left( X(1;x,y),t_{000} \right) dxdy,
\]
(where $Q',Q''$ are the left and right halves of $Q$).  We have $\nabla\cdot U\equiv 0$ because $U(x,y)=\nabla^\perp h(\varphi_{a,\delta}(x,y))$ for some function $h$, so $X(1;\cdot)$ is measure preserving and hence
\beq\lb{4.8}
\int_{X(1;Q')} \rho \left(x,y,t_{000} \right) dxdy = \int_{X(1;Q'')} \rho \left( x,y,t_{000} \right) dxdy.
\eeq
We would like to replace $X(1;Q'),X(1;Q'')$ by $Q',Q''$ here.  To control the resulting error, we need to estimate the size of the symmetric difference of $X(1;Q')$ and $Q'$ (see \eqref{4.10} below).

The definition of $T_{a,\delta}$, Lemma \ref{phi_a}(3), and \eqref{4.21} show that $\|u\|_\infty\le C$ for some $\delta$-independent $C<\infty$.  Since also $T_{a,\delta}(\varphi_{a,\delta}(x,y))^{-1}$ is bounded for $(x,y)$ near $\partial Q$, uniformly in $\delta\in(0,\tfrac 1{10})$ (because $\nabla\varphi_a\in L^\infty(Q)$), we obtain from \eqref{4.20}, the definition of $T_{a,\delta}$, and Lemma \ref{L.4.4}(2) that
\beq\lb{4.9}
|X(1;x,y)-(x,y)| \le C \delta^{(1-a)/(a+1)} \log \frac{2\delta} {\varphi_{a,\delta}(x,y)} \qquad \text{when $0<\varphi_{a,\delta}(x,y)\le\delta<\tfrac 1{10}$,}
\eeq
with $C$ independent of $\delta$ and $(x,y)$.  (Of course, $X(1;x,y)=(x,y)$ when $\varphi_{a,\delta}(x,y)>\delta$.)

We claim that this shows that if $D_l:=\{2^{-l}\delta<\varphi_{a,\delta}<2^{1-l}\delta\}$ for $l\ge 1$ (each $D_l$ is obviously invariant under the flows $u,\til u,U$), then
\beq\lb{4.10}
|\{(x,y)\in D_l \,:\, \text{exactly one of $(x,y)$ and $X(1;x,y)$ belongs to $Q'$} \}| \le c\delta^{2/(a+1)} \frac{l}{2^l}
\eeq 
for some $(\delta,l)$-independent $c<\infty$.  

Assume this is true.  Then $X(1;\cdot)$ being measure preserving and \eqref{4.8} show 
\[
\left|\int_{Q'} \rho(x,y,t_{000})dxdy-\frac 12 \int_{Q} \rho(x,y,0)dxdy \right| \le  \|\rho_0\|_\infty c\delta^{2/(a+1)} \sum_{l=1}^\infty \frac{l}{2^l} =: c' \delta^{2/(a+1)} \|\rho_0\|_\infty.
\]
Applying this for any $n\ge 0$ and any square $Q_{nij}$ (and then an analogous estimate involving $\til t_{nij}\in(n+\tfrac 12,n+1]$), as in the proof of Theorem \ref{T.4.3},  we see that if $\til Q$ is any of the four squares  with side length $2^{-(n+1)}$ which form $Q_{nij}$, then
\beq\lb{4.24}
\left| \fint_{\til Q} \rho(x,y,n+1)dxdy-\fint_{Q_{nij}} \rho(x,y,n)dxdy \right|\le 8c'\delta^{2/(a+1)}  \|\rho_0\|_\infty.
\eeq
  As a result, \eqref{4.12} now becomes (for mean-zero $\rho_0$)
\beq\lb{4.13}
\left|\fint_{Q_{nij}} \rho(x,y,n)dxdy \right| \leq 8c' \|\rho_0\|_\infty \sum_{k=0}^{n-1}\delta_k^{2/(a+1)}.
\eeq

We can now proceed as in the proof of Theorem \ref{T.4.3}, again with $n:=\tau_{\kappa,\eps}:=\lceil | \log_2 \tfrac{\kappa\eps} 2| \rceil + 2$.  When the flow is allowed to depend on $\eps$, we pick $\delta_k:=\delta := (\frac{\kappa}{16c'n})^{(a+1)/2}$, which then yields
\[
\sup_{t\in(0,n]} \|\nabla u(\cdot,t)\|_\infty \le C' \delta^{-\max\{1-a,2a\}/(a+1)} \le C'' \left(\frac{|\log(\kappa\eps)|}{\kappa}\right)^{\max\{1-a,2a\}/2} 
\]
for some $a$-dependent  $C',C''<\infty$.  We minimize the power (to $\tfrac 13$) by again choosing $a:=\tfrac 13$, and the first claim of the theorem follows.
 
If we want the flow to be independent of $\eps$, all is the same as in Theorem \ref{T.4.3}  but we need $\sum_{k=1}^\infty \delta_k^{2/(a+1)}\le \tfrac\kappa {16c'}$ instead of $\sum_{k=1}^\infty \delta_k\le \tfrac\kappa 8$.  We pick again $a:=\tfrac 13$, so that $\delta_{k-2}:=(\tfrac{\kappa}M k^{-1}(\log k)^{-2})^{2/3}$ satisfies this when $M:=16c'\sum_{k=2}^\infty k^{-1}(\log k)^{-2}$.  
Then the integral in \eqref{4.1a} is bounded by $C' \sum_{k=0}^{n-1} \delta_k^{-1/2}$, which is due to our choice of $n$ no more than $C\kappa^{-1/3} |\log(\kappa\eps)|^{4/3}\log|\log(\kappa\eps)|$, for some large enough universal $C$.

It remains to prove \eqref{4.10}.  The streamlines of $U$ are the level sets $\{\varphi_{a,\delta}=s\}$.  Each of them is a simple closed curve, so each ``moves'' in a single direction.  Also, $\{\varphi_{a,\delta}=s\}$ with  $s<\tfrac 1{10}$ intersects $\{x_0\}\times(0,\tfrac 12)$ in exactly one point when $x_0\in[\tfrac 14,\tfrac 34]$.  For $0<s<s'<\tfrac  1{10}$, let
\[
M_{s,s'}({x_0}) :=|\{\varphi_{a,\delta}(x,y)\in (s,s')\,:\, X(t;x,y)\in\{x_0\}\times(0,\tfrac 12) \text{ for some $t\in(0,1)$}  \}|
\]
be the measure of the set of those points between level sets $\{\varphi_{a,\delta}=s\}$ and $\{\varphi_{a,\delta}=s'\}$ which cross the segment $\{x_0\}\times(0,\tfrac 12)$ during time interval $(0,1)$ when advected by the flow $U$ (by the above, any point can cross at most once). Incompressibility of $U$ shows that $M_{s,s'}$ must be constant on $[\tfrac 14,\tfrac 34]$ for each $0<s<s'<\tfrac1{10}$.  We now pick $s:=2^{-l}\delta$ and $s':=2^{1-l}\delta$, and notice that the width of $D_l$ near $\{\tfrac 14\}\times(0,\tfrac 12)$ is easily shown to be  comparable to $2^{-l}\delta$.  Hence 
\[
M_{2^{-l}\delta,2^{1-l}\delta}\left(\frac 12\right)=M_{2^{-l}\delta,2^{1-l}\delta}\left(\frac 14\right) \le c \delta^{(1-a)/(a+1)} l \frac \delta {2^{l}}  = c\delta^{2/(a+1)} \frac{l}{2^{l}}
\]
by \eqref{4.9},
with some $(\delta,l)$-independent $c<\infty$.  Since the same bound is obtained if $M_{2^{-l}\delta,2^{1-l}\delta}$ is defined with $(\tfrac 12,1)$ in place of $(0,\tfrac 12)$, \eqref{4.10} follows if we replace $c$ by $2c$.
\end{proof}

The case $p\in[ \frac{3+\sqrt{5}}{2}, \infty)$ is almost identical to $p=\infty$, even simpler in a sense.  We now let
\[
\varphi_{a,\delta,p}(x,y)= \varphi_a(x,y)\, f \left(\frac{d_P(x,y)}{d_{a,\delta}} \right)^{1-a-\frac{2}{p}} \left(1-\log f\left(\frac{d_P(x,y)}{d_{a,\delta}} \right) \right)^{-1}
\]
and define $\psi_{a,\delta,p}$ via $\varphi_{a,\delta,p}$ as in the proof of Proposition \ref{p41}.  That proposition then holds for $\psi_{a,\delta,p}$, with a different bound in (2).  Indeed, essentially the same proof yields
\[
\|\nabla^2 \psi_{a,\delta,p}\|_{p} \leq C\delta^{-\max\left\{ 1-a-\tfrac 2p, 2a-\tfrac{a+1}{p}\right\}/(a+1)} \quad \text{ for } p>\frac{3+\sqrt{5}}{2},
\]
where the right-hand side is minimized (to $C\delta^{-(p^2-3p+1)/(2p^2-p)}$) when $a=\tfrac{p-1}{3p-1}$  (we fix this $a$ from now on). 
For $p=\frac{3+\sqrt{5}}{2}$ we get $\|\nabla^2 \psi_{a,\delta,p}\|_{p} \leq C |\log \delta|^{1/p}$ when $a=\frac{p-1}{3p-1}=\sqrt{5}-2$.

Lemma \ref{L.4.4} also holds for $p\in[ \frac{3+\sqrt{5}}{2}, \infty)$ and $a=\tfrac{p-1}{3p-1}$, with the estimate in part (2) being 
\[
\sup_{0<s\le\delta<1/10}  \delta^{-p/(2p-1)} | T_{\varphi_{a,\delta,p}}(s)- T_{\varphi_{a}}(\delta)|<\infty.
\]
Notice that there is no $s$-dependence in the first term when $p<\infty$, which means that in the argument from the proof of Theorem \ref{T.4.5} we do not need to split $\{\varphi_{a,\delta,p}\le\delta\}$ into the sets $D_l$ anymore.  That argument (which also uses the proof of Theorem \ref{T.4.3}) then yields
\[
\left|\fint_{Q_{nij}} \rho(x,y,n)dxdy \right| \leq 8c' \|\rho_0\|_\infty \sum_{k=0}^{n-1}\delta_k^{(3p-1)/(2p-1)}.
\]
instead of \eqref{4.13} (notice that $1+\tfrac{p}{2p-1}=\tfrac{3p-1}{2p-1}$ replaces $1+\tfrac{1-a}{a+1}=\tfrac 2{a+1}$).  Then the end of the proof of Theorem \ref{T.4.5} (before the proof of \eqref{4.10}) has $\delta_k^{2/(a+1)}$ and $\delta_k^{-1/2}$ replaced by $\delta_k^{(3p-1)/(2p-1)}$ and $\delta_k^{-(p^2-3p+1)/(2p^2-p)}$, respectively (the latter is $|\log \delta_k|^{1/p}$ when $p=\frac{3+\sqrt{5}}{2}$).  We therefore pick $\delta_{k}:=(\tfrac\kappa{16c'n})^{(2p-1)/(3p-1)}$ (or $\delta_{k-2}:=(\tfrac{\kappa}{M} k^{-1}(\log k)^{-2})^{(2p-1)/(3p-1)}$ in the $\eps$-independent case, with $M:=16c'\sum_{k=2}^\infty k^{-1}(\log k)^{-2}$)
and obtain the following.

\begin{theorem} \label{T.4.6}
For any mean-zero $\rho_0 \in L^\infty(Q)$ and any  $\kappa,\eps\in(0,\tfrac 12]$, there is an incompressible flow $u:Q \times \bbR^+ \to \mathbb{R}^2$ with  $u\cdot n = 0$ on $\partial Q\times\bbR^+$
 which $\kappa$-mixes $\rho_0$ to scale $\eps$ in a time $\tau_{\kappa,\eps}$ satisfying
\beq\lb{4.1b}
\int_0^{\tau_{\kappa,\eps}} \|\nabla u(\cdot,t)\|_p \,dt \le 
\begin{cases}
C_p \kappa^{-\nu_p} |\log(\kappa\eps)|^{1+\nu_p} & p\in(\frac{3+\sqrt{5}}{2},\infty),
\\ C_p |\log(\kappa \eps)|\, \left| \log \frac{|\log(\kappa\eps)|}{\kappa} \right|^{1/p} & p=\frac{3+\sqrt{5}}{2},
\end{cases}
\eeq
with $\nu_p := \frac{p^2-3p-1}{3p^2-p}$ and $C_p<\infty$ depending only on $p$.
The flow can be made independent of $\eps$, but the $p\in(\frac{3+\sqrt{5}}{2},\infty)$ alternative of the right-hand side of \eqref{4.1b} must be replaced by $C_p \kappa^{-\nu_p} |\log(\kappa\eps)|^{1+\nu_p}\log|\log(\kappa\eps)|$. 
\end{theorem}

\section{Mixing for periodic and no-slip boundary conditions} \lb{S5}

In this section we will show that the results from Sections 2--4 extend to periodic boundary conditions with only minor modifications to their proofs (in particular, ``perfect'' mixing is preserved here), and with some more work and slightly worse bounds also to no-slip boundary conditions.  
Let us start with the simpler case of periodic boundary conditions.

\begin{theorem} \lb{T.5.1}
Theorems  \ref{T.3.4}, \ref{T.4.5}, and \ref{T.4.6} hold when the no-flow boundary condition $u\cdot n = 0$ on $\partial Q\times\bbR^+$ is replaced by the periodic boundary condition.
\end{theorem}

\begin{proof}
Notice that the flows from all three theorems already satisfy periodic boundary conditions at all times $t> 1$.  Hence the only change required will be for $t\in (0,1]$.

First consider the case from Theorem \ref{T.3.4}.  What we need is that  $\int_{Q_{1ij}}\rho(x,y,1)dxdy=0$ for any $i,j\in\{0,1\}$.
We first let $u(x,y,t) = (2, 0)$ for $t\in(0,t_0]$ and  $u(x,y,t) = 0$ for $t\in(t_0,\tfrac 14]$, where $t_0\in[0,\tfrac 14]$ is such that $\int_{(0,1/2)\times(0,1)}\rho(x,y,t_0)dxdy = 0$
(which exists because the left and right halves of $Q$ would be swapped in time $\tfrac 14$.  Now let
\[
m_0:=\int_{Q_{100}}\rho(x,y,\tfrac 14)dxdy = - \int_{Q_{101}}\rho(x,y,\tfrac 14)dxdy,
\]
\[
m_1:=\int_{Q_{111}}\rho(x,y,\tfrac 14)dxdy = - \int_{Q_{110}}\rho(x,y,\tfrac 14)dxdy.
\]
For $t\in [\frac{1}{4}, \frac{3}{4})$ and $a:=\sqrt 5-2$ (the latter as in Theorem \ref{T.3.4}), define
\[
u (x,y,t)= 
\begin{cases}
\nabla^\perp \left[ (-1)^{i+j} 2^{-2} \psi_a\big(2x-i, 2y-j)\big) \right] & t\in (\tfrac 14,t_{ij}], 
\\ \nabla^\perp \left[ (-1)^{i+j} 2^{-2} \eta_a\big(2x-i, 2y-j)\big) \right] & t\in (t_{ij}, \tfrac 34], 
\end{cases}
\]
where $t_{ij} \in  [\frac{1}{4}, \frac{3}{4}]$ is such that the integrals of $\rho(\cdot,t_{ij})$ over the lower and upper halves of $Q_{1ij}$ are equal (then they are both $\tfrac 12 (-1)^{i+j} m_i$).
Finally, for $t\in (\frac{3}{4}, 1]$ we let $u(x,y,t) = (0,1)$, so that indeed $\int_{Q_{1ij}}\rho(x,y,1)dxdy=0$ for any $i,j\in\{0,1\}$, and we are done.


The case from Theorem \ref{T.4.5} is virtually identical, with $\psi_a,\eta_a$ replaced by $\psi_{a,\delta},\eta_{a,\delta}$ as in that theorem (i.e., $a=\tfrac 13$ and either $\delta= (\frac{\kappa}{16c'n})^{(a+1)/2}$ when $u$ can depend on $\eps$ or $\delta=\delta_1=(\tfrac{\kappa}M (\log 2)^{-2})^{2/3}$ when it cannot), and with $t_{ij} \in  [\frac{1}{4}, \frac{3}{4}]$ again chosen so that {\it if the flow $u$ in $Q_{ij}$ for $t\in(\tfrac 14,t_{ij}]$ were given as above but with $\til \psi_{a,\delta}$ in place of $\psi_{a}$, then  the integrals of $\rho(\cdot,t_{ij})$ over the lower and upper halves of $Q_{1ij}$ would be equal.}  This creates an error with the same bound as the error created in the proof of Theorem \ref{T.4.5} during time interval $[0,1]$.

Finally, the same adjustment works in the case from Theorem \ref{T.4.6}.
\end{proof}

The next theorem extends Theorem \ref{T.3.4} to no-slip boundary conditions. Here the difference is that our flow will not be a ``perfectly mixing'' one because $\rho(\cdot, n)$ may not have mean zero on the squares $Q_{nij}$ due to the no-slip condition.  We will have to control the resulting errors as we did in the theorems in Section \ref{S4}.  As a result, even though our ``mixing cost'' for the no-slip boundary conditions has the same dependence on $\eps$ as \eqref{2.1a} (i.e., $O(|\log \eps|)$), the dependence on $\kappa$ is worse.

\begin{theorem}
\label{T.5.2}
For any mean-zero $\rho_0 \in L^\infty(Q)$ and $\kappa\in (0,\tfrac 12]$, there is an incompressible $u:Q \times \bbR^+ \to \mathbb{R}^2$ with  $u = 0$ on $\partial Q\times\bbR^+$ such that  $u$ $\kappa$-mixes $\rho_0$ to any scale $\eps\in (0,\tfrac 12]$ in a time $\tau_{\kappa,\eps}$ satisfying
\beq \lb{2.1c}
\int_0^{\tau_{\kappa,\eps}} \|\nabla u(\cdot,t)\|_p \,dt \le C_p \left( |\log(\kappa\eps)| +  \kappa^{-1+1/p}\right)
\eeq
for each $p\in[1,\tfrac{3+\sqrt{5}}{2})$, with $C_p<\infty$ depending only on $p$.
 \end{theorem}

{\it Remark.} In particular, choosing $\kappa:=|\log \eps|^{-p/(p-1)}$ for a given $\eps\in(0,\tfrac 12]$  gives us an $\eps$-dependent flow that $|\log \eps|^{-p/(p-1)}$-mixes $\rho_0$ to scale $\eps$ in time $\tau_\eps$ such that \eqref{2.1c} holds with $\tau_\eps$ in place of $\tau_{\kappa,\eps}$ and $C_p |\log\eps|$ on the right-hand side.


\begin{proof}
  We will change the flow from no-flow case to no-slip  by multiplying the stream functions by a factor vanishing at $\partial Q$, which will make them vanish at $\partial Q$ to the second degree.

Let us therefore repeat  the construction from the proof of Theorem \ref{T.3.4} on each time interval $(n,n+\tfrac 12]$ (and then again on $(n+\tfrac 12, n+1]$), but with the following change.  Given $\rho(\cdot,n)$, 
let the stream function on each $Q_{nij}$ be
\[
\til\psi (x,y,t)= 
\begin{cases}
  (-1)^{i+j} 2^{-2n} \psi_a\big(2^{n} x - i, 2^n y - j\big)  & t\in (n, t_{nij}], 
\\   (-1)^{i+j} 2^{-2n} \eta_a\big(2^{n} x - i, 2^n y - j\big)  & t\in ( t_{nij},n+\tfrac 12], 
\end{cases}
\]
with $a=\sqrt 5-2$ and $t_{nij}$ chosen so that {\it we would have 
\beq\lb{5.10} 
\int_{Q_{nij}'} \rho(x,y,t_{nij})dxdy=\int_{Q_{nij}''} \rho(x,y,t_{nij})dxdy
\eeq 
 if the flow were $\til u:=\nabla^\perp\til\psi$ on the time interval $(n,t_{nij}]$}  (here again $Q_{nij}', Q_{nij}''$ are the left and right halves of $Q_{nij}$).  This is essentially the same construction as in Theorem \ref{T.3.4}, except that the two integrals need not equal 0 because we may have $\int_{Q_{nij}} \rho(x,y,n)dxdy\neq 0$.  

We now pick $\beta_n\in(0,\tfrac 12)$ (to be chosen later) and define
\[
 \psi(x,y,t) :=\tilde \psi(x,y,t) g_n(x,y):=\tilde \psi(x,y,t) f \left(2^n \beta_n^{-1} x(1-x) \right)  f \left(2^n \beta_n^{-1} y(1-y) \right),
\]
with $f$ from Lemma \ref{phi_a}, and use the flow $u := \nabla^\perp \psi$ for $t\in (n,n+\tfrac 12]$ instead (which satisfies the no-slip boundary condition). Notice that the $t_{nij}$ remain defined in terms of $\til u$.  

This means that we may not achieve \eqref{5.10} when $Q_{nij}$ touches $\partial Q$, but we can estimate the resulting error by finding the area of the set of streamlines of $\psi_a$ 
whose distance from $\partial Q$ is $<\beta_n$.  Indeed, if that area is $\gamma_n$, then $u=\til u$ on a subset of $Q_{nij}$ which is invariant under  $u$ over time interval $(n,t_{nij}]$ and has area $2^{-2n}(1-\gamma_n)$. Thus, for $Q_{nij}$ touching $\partial Q$ we will  have 
\[
\left| \int_{Q_{nij}'} \rho(x,y,t_{nij})dxdy-\int_{Q_{nij}''} \rho(x,y,t_{nij})dxdy \right| \le 2^{1-2n}\gamma_n \|\rho_0\|_\infty
\]
(the same argument appeared in Theorem \ref{T.4.3}), while for those not touching $\partial Q$ we will still have \eqref{5.10}.  The same statements then hold with $t_{nij}$ replaced by $n+\tfrac 12$, since  $Q_{nij}',Q_{nij}''$ are still invariant under $u$ on the time interval $(t_{nij},n+\tfrac 12]$.

After a similar argument is applied for $t\in (n+\tfrac 12, n+1]$ with the same $\beta_n$, we find that if $\til Q$ is any of the four squares of side-length $2^{-(n+1)}$ forming $Q_{nij}$, then 
\[
\left| \fint_{\til Q} \rho(x,y,n+1)dxdy-\fint_{Q_{nij}} \rho(x,y,n)dxdy \right|\le 4\gamma_n  \|\rho_0\|_\infty
\]
(and the difference is 0 if $Q_{nij}$ does not touch $\partial Q$).
  Since $\rho_0$ is mean-zero, it follows that
  \begin{equation}
  \left|\fint_{Q_{nij}} \rho(x,y,n)dxdy \right| \leq 4 \|\rho_0\|_\infty \sum_{k=0}^{n-1} \gamma_k \leq C \|\rho_0\|_\infty\sum_{k=0}^{n-1} \beta_k.
  \label{5.3}
  \end{equation}
   Here $C$ depends only on $a$ and is such that $4\gamma_n\le C\beta_n$.  This last estimate is proved as follows (with $C$ below depending only  on $a$ but changing from line to line).

Let $B_n = \{(x,y)\in Q: d_{\partial Q}(x,y) < \beta_n\}$ and  $s_n := \sup_{B_n}  \psi_a$. Then $\gamma_n = |\{\psi_a < s_n\}| $ by definition, so we need to show $|\{\psi_a < s_n\}| \leq C\beta_n$. We have $s_n \leq C \beta_n$ due to $\|\nabla  \psi_a\|_\infty <\infty$, and  the definition of $\psi_a$ gives $ \psi_a(x,y) \geq C^{-1} \min\{ d_{\partial Q}(x,y),  d_P(x,y)^{1+a}\}$, with $d_{\partial Q}$ the distance from $\partial Q$ and $d_P$ from Lemma~\ref{phi_a}. This finally yields (with a changing $C$)
\[
|\{ \psi_a < s_n\}  | \leq |\{ d_{\partial Q}(x,y) < Cs_n \}| + |\{ d_P(x,y) < Cs_n^{1/(1+a)} \}| \leq C(s_n+s_n^{2/(1+a)}) \leq C\beta_n.
\]

 Next we estimate $\|\nabla u(\cdot, t)\|_p$ for $t\in (n,n+\tfrac 12]$ (a similar estimate holds for $t\in (n+\tfrac 12,n+1]$). Recall that $u(x,y) = \tilde u(x,y)$ when $d_{\partial Q}(x,y)\ge 2^{-n}\beta_n$. On the rest of $Q$ 
 we have
 \[
 \begin{split}
|\nabla  u| & \leq |\nabla \tilde u| g_n + 2 |\nabla \tilde \psi| \, |\nabla g_n| + |\tilde \psi| \, |\nabla^2 g_n| \leq  |\nabla \tilde u| + C\beta_n^{-1},
\end{split}
\]
(with an $a$-dependent $C$),
where in the last inequality  we used 
\[
|\tilde\psi(x,y)| \leq 2^{-n}(\|\nabla \psi_a\|_\infty+ \|\nabla \eta_a\|_\infty) d_{\partial Q} (x,y) \leq C2^{-2n}\beta_n
\]
 when  $d_{\partial Q}(x,y)< 2^{-n}\beta_n$. Since $|\{d_{\partial Q}<2^{-n}\beta_n\} |\leq 2^{2-n} \beta_n$ and we have $\sup_{t>0} \|\nabla \tilde u(\cdot,t)\|_p\le C_p'$ as in the proof of Theorem \ref{T.3.4}, we now obtain for $p\in [1,\tfrac {3+\sqrt 5}2)$ (and a new $C_p'$),
\begin{equation}
\sup_{t\in (n,n+1]} \|\nabla  u(\cdot, t)\|_p \leq C_p' \left(1+2^{-n/p} \beta_n^{(1-p)/p} \right) .
\label{u_p}
\end{equation}

Given $\kappa,\eps\in(0,\tfrac 12]$, let again $n:=\tau_{\kappa,\eps}:=\lceil | \log_2 \tfrac{\kappa\eps} 2| \rceil + 2$.  As in the proof of Theorem \ref{T.4.3}, the constructed flow $u$ will $\kappa$-mix $\rho_0$ to scale $\eps$ in time $\tau_{\kappa,\eps}$, provided $ \sum_{k=0}^{n-1} \beta_k\le \tfrac \kappa {2C}$ (with $C$ from \eqref{5.3}).  
We pick $\beta_{k-2} := \tfrac \kappa {M}k^{-1}(\log k)^{-2}$, with  $M:=2C\sum_{k=0}^\infty k^{-1}(\log k)^{-2}$, which then yields (with $C_p''<\infty$ depending only on $p$)
\[
\int_0^{\tau_{\kappa,\eps}} \|\nabla u(\cdot,t)\|_p\, dt \le C_p' \sum_{k=0}^{n-1} \left(1+2^{-k/p} \beta_k^{(1-p)/p} \right)  \le C_p'n + C_p'' \kappa^{(1-p)/p}
\]
for $p\in [1,\tfrac {3+\sqrt 5}2)$.  The result follows because the $\beta_k$ are independent of $\eps$, hence so is $u$.
\end{proof}

The following two results extend Theorems \ref{T.4.5} and \ref{T.4.6} to no-slip boundary conditions. 
%

\begin{theorem}
\label{T.5.3}
For any mean-zero $\rho_0 \in L^\infty(Q)$ and any  $\kappa,\eps\in(0,\tfrac 12]$, there is an incompressible flow $u:Q \times \bbR^+ \to \mathbb{R}^2$ with  $u= 0$ on $\partial Q\times\bbR^+$
 which $\kappa$-mixes $\rho_0$ to scale $\eps$ in a time $\tau_{\kappa,\eps}$ satisfying
\beq \lb{4.1c}
\int_0^{\tau_{\kappa,\eps}} \|\nabla u(\cdot,t)\|_\infty \,dt \le C \kappa^{-1} |\log(\kappa\eps)|^2,
\eeq
 with a universal $C<\infty$.
The flow can be made independent of $\eps$ if the right-hand side of \eqref{4.1c} is replaced by $C\kappa^{-1} |\log(\kappa\eps)|^2 (\log|\log(\kappa\eps)|)^2$. 
\end{theorem}

\begin{proof} 
This is almost the same as the proof of Theorem \ref{T.5.2}, except that $\tilde \psi$ is constructed not using $\psi_a$ and  $\eta_a$ but the corresponding stream functions from the proofs of Theorems \ref{T.4.3} and \ref{T.4.5}. This results in the following two changes relative to the proof of Theorem \ref{T.5.2}.

The first change is that \eqref{5.3} will also include the error from \eqref{4.13},
hence the estimate becomes (with $a:=\tfrac 13$, so $\tfrac 2{a+1}=\tfrac 32$)
\begin{equation}
  \left|\fint_{Q_{nij}} \rho(x,y,n)dxdy \right| \leq C \|\rho_0\|_\infty  \sum_{k=0}^{n-1} \left( \beta_k +  \delta_k^{3/2} \right).
  \label{5.6}
\end{equation}
The second change results from the flow in Theorem \ref{T.4.5} satisfying $\|\nabla u(\cdot, t)\|_\infty \leq C\delta_n^{-1/2}$ for $t\in (n,n+1]$,  
so \eqref{u_p}  becomes 
\[
\sup_{t\in (n,n+1]}  \|\nabla  u(\cdot, t)\|_\infty \leq C' \left( \beta_n^{-1} + \delta_n^{-1/2} \right).
\]
 
The rest of the proof follows that of Theorem \ref{T.5.2}, again with $n:=\tau_{\kappa,\eps}:=\lceil | \log_2 \tfrac{\kappa\eps} 2| \rceil + 2$. When our flow is allowed to depend on $\eps$,  we let $\beta_{k}:= \tfrac\kappa{4Cn}$ and $\delta_{k}:=(\tfrac\kappa{4Cn})^{2/3}$.  Thus $u$ $\kappa$-mixes $\rho_0$ to scale $\eps$ in time $\tau_{\kappa,\eps}$, and we have
\[
\int_0^{\tau_{\kappa,\eps}} \|\nabla u(\cdot,t)\|_\infty\, dt \le C' \sum_{k=0}^{n-1} \left( \beta_k^{-1} + \delta_k^{-1/2} \right)  = C' n\left[ \left(\frac\kappa{4Cn} \right)^{-1} + \left(\frac\kappa{4Cn} \right)^{-1/3}\right] \le \frac{8C'Cn^2}\kappa,
\]
which is bounded by $C\kappa^{-1}|\log(\kappa\eps)|^2$ (with a new  $C$).   If we instead want the flow to be independent of $\eps$, we pick $\beta_{k-2}:=\tfrac{\kappa}M k^{-1}(\log k)^{-2}$ and $\delta_{k-2}:=(\tfrac{\kappa}M k^{-1}(\log k)^{-2})^{2/3}$, with $M:=4C\sum_{k=2}^\infty k^{-1}(\log k)^{-2}$. The above estimate then gains a factor of $(\log|\log(\kappa\eps)|)^2$.  
\end{proof}

{\it Remark.}  This result is the same as the one in the remark at the beginning of Section \ref{S4} for $p=\infty$. The method is  different, though, which will make a difference for $p<\infty$ below.

\begin{theorem} \label{T.5.4}
Theorem \ref{T.4.6} holds when the no-flow boundary condition $u\cdot n = 0$ on $\partial Q\times\bbR^+$ is replaced by the no-slip boundary condition $u = 0$ on $\partial Q\times\bbR^+$, and $C_p\kappa^{-1+1/p}$ is added to the right-hand side of \eqref{4.1b} in both the $\eps$-dependent and $\eps$-independent cases.
 \end{theorem}

\begin{proof} 
We proceed in the same way as in Theorem \ref{T.5.3}, but estimates for $p=\infty$ from Theorem \ref{T.4.5} are replaced by the corresponding estimates for $p\in [\frac{3+\sqrt{5}}{2}, \infty)$ from Theorem \ref{T.4.6}. This ultimately yields a flow $u$ for which
\[
  \left|\fint_{Q_{nij}} \rho(x,y,n)dxdy \right| \leq C \|\rho_0\|_\infty  \sum_{k=0}^{n-1} \left(\beta_k +  \delta_k^{(3p-1)/(2p-1)} \right)
\]
for all $n$, and with $n:=\tau_{\kappa,\eps}:=\lceil | \log_2 \tfrac{\kappa\eps} 2| \rceil + 2$ we also have
\[
\int_0^{\tau_{\kappa,\eps}}  \|\nabla  u(\cdot, t)\|_p \, dt\leq C_p' \sum_{k=0}^{n-1}\left( 2^{-k/p} \beta_k^{(1-p)/p} + \delta_k^{-(p^2-3p+1)/(2p^2-p)} \right)
\]
(with $|\log \delta_k|^{1/p}$ in place of $\delta_k^{-(p^2-3p+1)/(2p^2-p)}$ when $p=\frac{3+\sqrt{5}}{2}$).  

We now take  $\beta_k$ from the proof of Theorem \ref{T.5.2} (with $4C$ in place of $2C$) and $\delta_{k}$ from the proof of Theorem \ref{T.4.6} (with $4C$ in place of  $16c'$).
We thus obtain that $u$ $\kappa$-mixes $\rho_0$ to scale $\eps$ in time $\tau_{\kappa,\eps}$, and also
\beq\lb{5.15}
\int_0^{\tau_{\kappa,\eps}}  \|\nabla  u(\cdot, t)\|_p \, dt\leq  
C_p''\kappa^{(1-p)/p} + C_p' \left(\frac\kappa{4C}\right)^{-\nu_p} n^{1+\nu_p}
\eeq
for $p>\tfrac{3+\sqrt 5}2$ (with the last term being $C_p'(\tfrac\kappa{M})^{-\nu_p} n^{1+\nu_p} ( \log n)^{2\nu_p}$ in the $\eps$-independent case; notice that $\nu_p\le \tfrac 13$).  For $p=\tfrac{3+\sqrt 5}2$ the last term in \eqref{5.15} is $C_p' n |\log \tfrac\kappa{n} |^{1/p}$
in both the $\eps$-dependent and $\eps$-independent cases.  This yields the result.
\end{proof}

\section{Un-mixing} \lb{S6}

\begin{proof}[Proof of Theorem \ref{T.6.1}]
We will prove an equivalent statement, with $\rho(\cdot,n)=\chi_B$ and
\beq\lb{6.1a}
\sup_{t\in(0,n]} \|\nabla u(\cdot,t)\|_p \le C \kappa^{-1+1/p}n^{1-1/p} \qquad\text{for each $p\in[1,\infty]$}
\eeq
(the equivalence is obtained by changing this flow to $nu(\cdot,\tfrac tn)$).
Let $\tht_{ij}:=2^{2n}|A\cap Q_{nij}|$.
We claim that it is sufficient to find incompressible  $\til u:Q \times (0,n] \to \mathbb{R}^2$ with  $\til u = 0$ on $\partial Q\times\bbR^+$ and satisfying \eqref{6.1a} such that the solution to \eqref{1.1} with $\til \rho(\cdot,0)=\chi_{\{0<x<|A|\}}$ satisfies 
$\|\til \rho(\cdot,n)-\chi_S\|_1\le\kappa$, for some $S\subseteq Q$ with
$2^{2n}|S\cap Q_{nij}|= \tht_{ij}$ for any $i,j\in\{0,\dots,2^n-1\}$.  Indeed, then $\int_{Q_{nij}} |\chi_S-\chi_A| \,dxdy\le2^{1-2n}\kappa$ whenever $\tht_{ij}\notin(\kappa,1-\kappa)$. Combining this with the hypothesis that at most $2^{2n}\kappa$ of the $\theta_{ij}$ belong to $(\kappa, 1-\kappa)$, we obtain $\| \chi_S-\chi_A \|_1\le 3\kappa$.  It therefore suffices to let $u(x,y,t):=-\til u(x,y,n-t)$.  We now have that $\til \rho(\cdot,n-t)$ also solves \eqref{1.1},  so incompressibility of $u$ and $|A|=|B|$ yield  
\[
2\left|B\cap \left[(|A|,1)\times(0,1)\right] \right|=  \| \til\rho(\cdot,0)-\rho(\cdot,n) \|_1 = \| \til\rho(\cdot,n)-\chi_A \|_1 \le 4\kappa.
\]
 This proves the result. Hence it suffices to prove the following lemma (with the $\til{}$ dropped).
\end{proof}

\begin{lemma} \label{L.6.2}
There is $C>0$ such that for any $n\ge 1$, $\kappa\in(0,\tfrac 12]$, and $\tht_{ij}\in[0,1]$ (with $i,j\in\{0,\dots,2^n-1\}$), there is $S\subseteq Q$ with $2^{2n}|S\cap Q_{nij}|= \tht_{ij}$ for any $i,j\in\{0,\dots,2^n-1\}$, and there is an incompressible flow $u:Q \times (0,n] \to \mathbb{R}^2$ with  $u = 0$ on $\partial Q\times(0,n]$ and satisfying \eqref{6.1a}, such that the solution to \eqref{1.1} with $\rho(\cdot,0)=\chi_{\{0<x<|S|\}}$ satisfies 
$\| \rho(\cdot,n)-\chi_S\|_1\le\kappa$.
\end{lemma}

{\it Remark.}  Remark 2 after Theorem \ref{T.6.1} applies here, too.

\begin{proof}
This is related to the previous constructions, particularly to that in the remark at the beginning of Section \ref{S4}.  The basic stream function here will be 
\beq\lb{6.2}
 \psi^\del(x,y) :=\psi(x,y) g^\del(x,y):=\psi(x,y) f \left(\del^{-1} x(1-x) \right)  f \left(\del^{-1} y(1-y) \right),
\eeq
with $\psi$ from Proposition \ref{prop:stream},  $f$ from Lemma \ref{phi_a}, and $\del\in (0,1]$.
Notice that if $\del>0$, then $\nabla^\perp \psi^\del=0$ on $\partial Q$ since $\psi=0$ on $\partial Q$ and $f(0)=0$.  Then \eqref{eq:dpsi},  Lemma \ref{lemma:phi}(1,3,4), \eqref{eq:d2phi}, \eqref{eq:min_dphi}, and the definition of $f$ show for some $\del$-independent $C<\infty$,
 \[
|\nabla^2  \psi^\del|  \leq |\nabla^2 \psi| g^\del + 2 |\nabla \psi| \, |\nabla g^\del| +  \psi \, |\nabla^2 g^\del| \leq  
\begin{cases} 
Cd_c(x,y)^{-1} & d_{\partial Q}(x,y)\ge \del, 
\\  C\del^{-1} & d_{\partial Q}(x,y)<\del,
\end{cases}
\]
where we also used $\psi(x,y) \leq \|\nabla \psi\|_\infty d_{\partial Q} (x,y)$ (and $d_c,d_{\partial Q}$ are the distances from the nearest corner and from $\partial Q$, respectively).  This yields for any $p\in[1,\infty]$ and $\delta\in (0,1]$,
\beq \lb{6.3}
\|\nabla^2\psi^\del\|_p\le C\del^{-1+1/p},
\eeq 
with a new $p$-independent $C$.
 
Let us  
define 
\[
\tht_0:= 2^{-2n} \sum_{i=0}^{2^{n-1}-1} \sum_{j=0}^{2^n-1} \tht_{ij}\le \frac 12 \qquad\text{and}\qquad \tht_1:= 2^{-2n} \sum_{i=2^{n-1}}^{2^n-1} \sum_{j=0}^{2^n-1} \tht_{ij} \le \frac 12.
\]
Notice that if $S$ is as in the statement of the lemma, then $\tht_0=|S\cap [(0,\tfrac 12)\times(0,1)]|$ and $\tht_1=|S\cap [(\tfrac 12,1)\times(0,1)]|$.

\begin{figure}[htbp]
\includegraphics[scale=1.2]{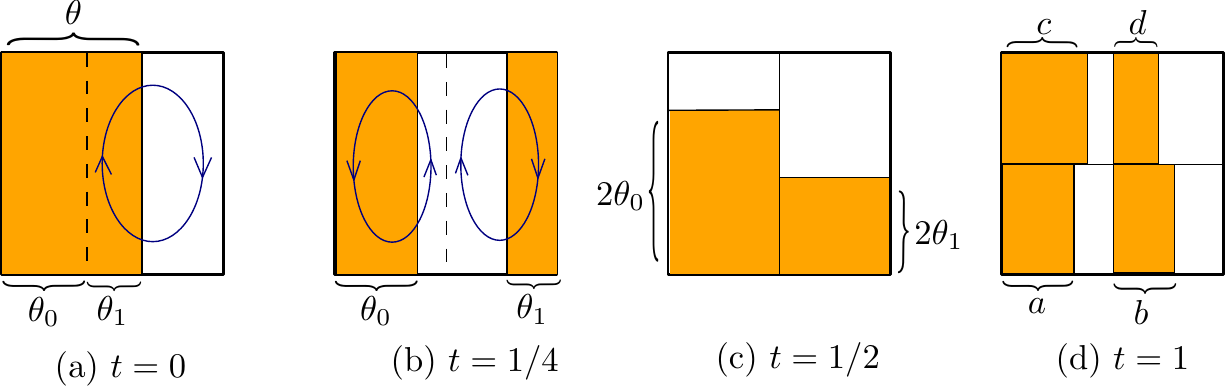}
\caption{What $\rho(\cdot, t)$ would be at different times if $\psi^{\delta}$ were replaced by $\psi$ in the construction of the flow $u$. }\label{fig_unmix}
\end{figure}

For $t\in(0,\tfrac 14]$ let $u(x,y,t)=0$ when $x\in(0, \tht_0]$ and 
\beq\lb{6.6}
u(x,y,t):=\nabla^\perp \left[ 2(1-\tht_0) \psi^\del \left( \frac{x-\tht_0}{1-\tht_0},y \right) \right]
\eeq
when $x\in(\tht_0,1)$.  This flow, as well as $\rho(\cdot,0)$, is illustrated in Figure \ref{fig_unmix}(a).
Proposition \ref{prop:stream}(3) and symmetry show that {\it if $\psi^\del$ were replaced by $\psi$ here, then $\rho(\cdot,\tfrac 14)$ would equal  $\chi_{\{0<x<\tht_0\}} + \chi_{\{1-\tht_1<x<1\}}$}, as illustrated in Figure \ref{fig_unmix}(b).  This is because the flow $\nabla^\perp [2\psi]$ rotates $Q$ by $180^\circ$ in time $\tfrac 14$.  However, the factor $g^\del$ in \eqref{6.2} limits us to only
\beq\lb{6.4}
\left\| \rho(\cdot,\tfrac 14) - \chi_{\{0<x<\tht_0\}} - \chi_{\{1-\tht_1<x<1\}} \right\|_1 \le C\del
\eeq
for a new $C<\infty$.  This is due to the area of the set of streamlines of $\psi$ whose distance from $\partial Q$ is $<\del$ (which are the ones affected by $g^\del$) being bounded by $C\del$ (by $\nabla\varphi\in L^\infty$ and \eqref{eq:min_dphi}).  Notice also that $\tht_0\le \tfrac 12$ and \eqref{6.3} show (with a new $C$)
\beq \lb{6.5}
\sup_{t\in(0,1/4]} \|\nabla u(\cdot, t)\|_p\le C\del^{-1+1/p}.
\eeq


For $t\in(\tfrac 14,\tfrac 12]$ let 
\[
u(x,y,t):=\nabla^\perp \left[ \frac {(-1)^{\lfloor 2x \rfloor+1}}2\psi^\del \left( \{2x\} ,y \right) \right],
\]
with $\lfloor 2x \rfloor, \{2x\}$ the integer and fractional parts of $2x$, respectively.  This flow is illustrated in Figure \ref{fig_unmix}(b).  Again \eqref{6.5} holds for $t\le \tfrac 12$, and again,  {\it if $\psi^\del$ were replaced by $\psi$ here and in \eqref{6.6}, then $\rho(\cdot,\tfrac 12)$ would equal}  
\[
\chi_{(0,1/2)\times(0,2\tht_0)} + \chi_{(1/2,1)\times(0,2\tht_1)},
\]  
as shown in Figure \ref{fig_unmix}(c).
This is because the flow $\nabla^\perp \psi$ rotates $Q$ clockwise (and $\nabla^\perp [-\psi]$ rotates $Q$ counter-clockwise) by $90^\circ$ in time $\tfrac 14$, due to Proposition \ref{prop:stream}(3) and symmetry.  However, as above,  the extra factor $g^\delta$ means we only obtain (with a new $C$)
\[
\left\| \rho(\cdot,\tfrac 12) - \chi_{(0,1/2)\times(0,2\tht_0)} - \chi_{(1/2,1)\times(0,2\tht_1)} \right\|_1 \le C\del.
\]

For $t\in(\tfrac 12,1]$ we  run the same argument as for $t\in(0,\tfrac 12]$, but separately on the rectangles $(0,\tfrac 12)\times (0,1)$ and $(\tfrac 12,1)\times (0,1)$ instead of $Q$, and rotated by $90^\circ$ (so the first argument of the stream functions is either $2x$ or $2x-1$, instead of the second being $y$).  The end result is 
\[
\left\| \rho(\cdot,1) - \chi_{(0,a)\times(0,1/2)} - \chi_{(1/2,1/2+b)\times(0,1/2)} - \chi_{(0,c)\times(1/2,1)} - \chi_{(1/2,1/2+d)\times(1/2,1)} \right\|_1 \le C\del,
\]
with a new $C$ and
\[
a:= 2^{1-2n} \sum_{i,j=0}^{2^{n-1}-1}  \tht_{ij}, 
\quad b:= 2^{1-2n} \sum_{i=2^{n-1}}^{2^n-1} \sum_{j=0}^{2^{n-1}-1} \tht_{ij}, 
\quad c:= 2^{1-2n} \sum_{i=0}^{2^{n-1}-1} \sum_{j=2^{n-1}}^{2^n-1} \tht_{ij}, 
\quad d:= 2^{1-2n} \sum_{i,j=2^{n-1}}^{2^n-1} \tht_{ij}
\]
(so  $a+c=2\tht_0$ and $b+d=2\tht_1$).
That is, $\rho(\cdot, 1)$ is $C\delta$ close in $L^1(Q)$ to the function from Figure \ref{fig_unmix}(d), which is the sum of characteristic functions of four rectangles, each being 
the intersection of one of the squares $Q_{100},Q_{101},Q_{110},Q_{111}$ 
with a half-plane with normal vector $(1,0)$, and each having area equal to $2^{-2n}$ times the sum of those $\tht_{ij}$ for which $Q_{nij}$ lies inside that square.

It is clear that a scaled version of this construction can be repeated for $t\in(1,2]$, separately on each square $Q_{100},Q_{101},Q_{110},Q_{111}$, with \eqref{6.5} still valid for these $t$ (because scaling of the stream functions is by a factor of 2 in space and $\tfrac 14$ in value) and at the expense of an additional $L^1$-norm error $\tfrac14 C\del$ on each square.  Continuing up to time $t=n$ and scale $2^{-n}$, we obtain $\|\rho(\cdot, n)-\sum_{i,j=0}^{2^n-1} \chi_{R_{ij}}\|_1\le Cn\del$, where $R_{ij}=(2^{-n}i,2^{-n}(i+\tht_{ij}))\times(2^{-n}j,2^{-n}(j+1))\subseteq Q_{nij}$ 
for any $i,j\in\{0,\dots,2^n-1\}$.  So $S:=\bigcup_{i,j=0}^{2^n-1} R_{ij}$ and it now suffices to pick $\del:=\tfrac\kappa{Cn}$.

The claim of the remark follows by replacing $\psi^\del$ by $\psi$ in the proof.
\end{proof}

\begin{proof}[Proof of Corollary \ref{C.6.3}]
Let $\tht_{nij}:=2^{2n}|A\cap Q_{nij}|$ and consider the setting from the last sentence of the previous proof (i.e., $\kappa=0$ and $\psi$ in place of $\psi^\del$) for any $n$ and with the $\tht_{ij}$ being the $\tht_{nij}$.  Then the constructed flows for $n=n_1$ and $n=n_2$ obviously coincide for $t\in(0,\min\{n_1,n_2\}]$.  Thus there is a unique incompressible flow $u:Q\times \bbR^+\to\bbR^2$ satisfying the no-flow boundary condition which coincides with all these flows (for different $n$) on their time intervals of definition.  One easily sees that $\sup_{t>0} \|u(\cdot, t)\|_{BV}<\infty$ so, in particular, \eqref{1.1} is well-posed.

Due to the nature of the scaling of the stream functions in the previous proof by a factor of $2^k$ in space and $2^{-2k}$ in value for $t\in (k,k+1]$ (relative to $t\in (0,1]$), we have
\[
\sup_{t\in(k,k+1]} \|u(\cdot,t)\|_{\dot W^{s,p}} \le C_{s,p} 2^{(s-1)k}
\]
for any $s<\tfrac 1p$ and  $k=0,1\dots$ (we need  $s<\tfrac 1p$ to make $u(\cdot,t)\in W^{s,p}$, since it is discontinuous along finitely many lines).  Indeed, this follows from Lemma \ref{sobolev_1} in the appendix (rather than from interpolation, as in the introduction, because the flows here do not belong to $W^{1,p}(Q)$).  Scaling $u|_{(k,k+1]}$ in time by a factor of $(1-2^{s-1})^{-1}2^{(1-s)k}$ and multiplying it by the same factor creates an incompressible flow $\til u$ on $Q\times (0,1)$ such that 
\[
\sup_{t\in(0,1)} \|\til u(\cdot,t)\|_{\dot W^{s,p}} \le C_{s,p} 
\]
(new $C_{s,p}$ equals old $C_{s,p}$ times $(1-2^{s-1})^{-1}$) and
the solution to \eqref{1.1} with $\til u$ in place of $u$, and  with $\rho(\cdot,0)=\chi_{\{0<x<|A|\}}$ satisfies $\fint_{Q_{nij}} \rho(x,y,1-2^{(s-1)n})\,dxdy=\tht_{nij}$ for all $n,i,j$.  Because $A$ is measurable (so a.e.~$(x,y)\in Q$ is its Lebesgue point), we obtain  $\lim_{t\to 1} \|\rho(\cdot,t)-\chi_A\|_1=0$.
\end{proof}


\medskip
\appendix
\section*{Appendix:  Properties of  stream functions $\varphi,\varphi_a,\varphi_{a,\delta}$ and two inequalities} \lb{SA1}
\renewcommand{\theequation}{A.\arabic{equation}}
\renewcommand{\thetheorem}{A.\arabic{theorem}}
\setcounter{theorem}{0}
\setcounter{equation}{0}

%
%

\begin{proof}[Proof of Lemma \ref{lemma:phi}]
The first two claims of (1) are obvious, and the other two follow from
\begin{equation}
 \varphi_x(x,y) =  4\cos(\pi x) \left(1+\dfrac{\sin(\pi x)}{\sin(\pi y)}\right)^{-2}
\label{eq:dx_phi}
\end{equation}
and a similar expression for $\varphi_y$.

To show (2), we start by taking another $x$-derivative of \eqref{eq:dx_phi}:
\[
 \varphi_{xx}(x,y) = -4\pi \sin(\pi x)  \left(1+\dfrac{\sin(\pi x)}{\sin(\pi y)}\right)^{-2} - 8\pi \cos^2(\pi x)  \left(1+\dfrac{\sin(\pi x)}{\sin(\pi y)}\right)^{-2} \left(\sin(\pi x) + \sin(\pi y)\right)^{-1}.
\]
Hence 
\[
| \varphi_{xx}(x,y) | \leq 4\pi + \frac {8\pi}{\sin(\pi x) + \sin(\pi y)} \leq 4\pi+ \frac{8\pi}{\sqrt{2}\,d_c(x,y)} \le  \frac{40} {d_c(x,y)},
\]
with $d_c$ the distance to the closest of the four corners of $Q$.
Obviously, $ \varphi_{yy}$ obeys the same bound due to symmetry. As for the cross term, taking the $y$ derivative of \eqref{eq:dx_phi} gives
\[
 \varphi_{xy}(x,y) = 8\pi \cos(\pi x)\cos(\pi y) \left(1+\dfrac{\sin(\pi x)}{\sin(\pi y)}\right)^{-1} \left(1+\dfrac{\sin(\pi y)}{\sin(\pi x)}\right)^{-1}    \left(\sin(\pi x) + \sin(\pi y)\right)^{-1},
\]
hence 
$| \varphi_{xy}(x,y)| \leq 20 d_c(x,y)^{-1}.$ These estimates now show (2) because they yield
\begin{equation}
|\nabla^2 \varphi(x,y)| \leq \frac{70}{ d_c(x,y)}.
\label{eq:d2phi}
\end{equation}

To show (3), first note that $\partial_n  \varphi = -4$ on $\partial Q_c$ implies $T_\varphi(0)=1$. 
Consider the triangle $Q_T := \{0< y< x <\tfrac 12\}$, and for $s\in(0,\tfrac 2\pi)$ let  $\Gamma(s) := \{(x,y)\in Q_T: \varphi(x,y)=s\}$ (note that $0\le\varphi\le\tfrac 2\pi$ on $Q$). By symmetry we have
$$T_\varphi(s) = 8 \int_{\Gamma(s)} \frac{1}{|\nabla \varphi|} d\sigma,$$
so with $|\Gamma(s)|$ the length of the curve $\Gamma(s)$,
\begin{equation}
T_\varphi(s) \leq 8 |\Gamma(s)| \, \max_{\Gamma(s)} |\nabla \varphi|^{-1}.
\label{A3}
\end{equation}

Since  $\varphi_y \geq \varphi_x > 0$ on $Q_T$, 
the function of $x$ whose graph is $\Gamma(s)$ has slope between $-1$ and 0 on its domain
$(b(s),\tfrac 12)$, with $b(s) := \frac{1}{\pi}\arcsin\tfrac{\pi s}2$.
Thus 
\begin{equation}|\Gamma(s)| \leq \sqrt{2}\left(\frac{1}{2} - \frac{1}{\pi}\arcsin\frac{\pi s}2\right) \quad(\le 1),
\label{eq:gamma_bound}
\end{equation}
for all $s$, and $\lim_{s\to 2/\pi} |\Gamma(s)| = 0$.
Moreover, on $\Gamma(s)$ we have $\sin(\pi x)\ge\sin(\pi y)$, so
\begin{equation}
\begin{split}
|\nabla \varphi|^{-1} &\leq (\varphi_y)^{-1} \leq \frac{1}{4\cos(\pi y)} \left(1+\frac{\sin(\pi y)}{\sin(\pi x)}\right)^{2}  \leq \frac{1}{\cos(\pi b(s))} = \frac{1}{\sqrt{1-(\pi s/2)^2}}
\label{eq:min_dphi}
\end{split}
\end{equation}
there. 
Combining \eqref{eq:gamma_bound} and \eqref{eq:min_dphi} with \eqref{A3} now gives 
\[
T_\varphi(s) \leq  8 \sqrt 2 \, \frac{\frac{1}{2}-\frac{1}{\pi} \arcsin \frac{\pi s}2}{\sqrt{1-\left(\frac{\pi s}2\right)^2}}.
\]
The fraction is  bounded in $s\in [0,\frac{2}{\pi})$ away from $\tfrac 2\pi$, and one can easily check that it converges to $\tfrac{1}{\pi}$ as $s\to \frac{2}{\pi}$. This proves (3).

It remains to prove (4). Let us first use the divergence theorem to find
\[
\begin{split}
T_\varphi(s) = \int_{\{\varphi=s\}} -n\cdot \frac{\nabla \varphi}{|\nabla \varphi|^2} d\sigma = \int_{\{\varphi>s\}} -\nabla\cdot \frac{\nabla \varphi}{|\nabla \varphi|^2} dxdy,
\end{split}
\]
with $n$ the outer unit normal vector to $\{\varphi>s\}$. This yields
\begin{equation}
T'_\varphi(s) = \int_{\{\varphi=s\}} \frac{1}{|\nabla \varphi|} \nabla\cdot \frac{\nabla \varphi}{|\nabla \varphi|^2} d\sigma = \int_{\{\varphi=s\}} \left( \frac{\Delta \varphi}{|\nabla \varphi|^3} - 2\frac{\nabla \varphi \cdot( \nabla^2 \varphi \nabla \varphi) }{|\nabla \varphi|^5} \right) d\sigma
\label{T''}
\end{equation}
for $s\in(0,\tfrac 2\pi)$, so $T_\varphi$ is differentiable on $(0,\tfrac 2\pi)$. We also obtain
\begin{equation}
|T'_\varphi(s)| \max_{\{\varphi=s\}} |\nabla \varphi|^2 \leq 4\int_{\{\varphi=s\}} \frac{|\nabla^2 \varphi|}{|\nabla \varphi|^3 }d \sigma   \left(\max_{\{\varphi=s\}} |\nabla \varphi|^2 \right).
\label{eq:bound_T'}
\end{equation}
Recall that $\|\nabla \varphi\|_\infty<\infty$ by (1), and $|\nabla \varphi|^{-1}$ is easily seen to be uniformly bounded away from $(\tfrac 12,\tfrac 12)$ (where $s\approx \frac{2}{\pi}$).  Since \eqref{eq:d2phi} shows that $|\nabla^2 \varphi|$ is uniformly bounded away from the corners of $Q$ (where $s\approx 0$),  we only need to bound the RHS of \eqref{eq:bound_T'} near $s=0,\frac{2}{\pi}$.

For $s$ close to $\frac{2}{\pi}$ we have $|\nabla^2 \varphi| \leq 200$ by \eqref{eq:d2phi}, hence
\beq \lb{3.1}
|T'_\varphi(s)| \max_{\{\varphi=s\}} |\nabla \varphi|^2 \leq 800 \int_{\{\varphi=s\}} \frac{1}{|\nabla \varphi|}d \sigma \frac{ \max_{\{\varphi=s\}} |\nabla \varphi|^2}{\min_{\{\varphi=s\}} |\nabla \varphi|^2}
= 800 T_\varphi(s) \frac{ \max_{\{\varphi=s\}} |\nabla \varphi|^2}{\min_{\{\varphi=s\}} |\nabla \varphi|^2}.
\eeq
The RHS is bounded near $\tfrac 2\pi$ because $T_\varphi$ is bounded and the fraction is bounded near $\tfrac 2\pi$.  Indeed, by symmetry it suffices to show the latter with $\Gamma(s)$ in place of $\{\varphi=s\}$.  We have 
\[
|\nabla \varphi| \geq \cos(\pi b(s)) = \sin\left( \pi\left( \frac{1}{2}-b(s) \right) \right)
\]
 on $\Gamma(s)$ by \eqref{eq:min_dphi}.  We also have
$|\nabla \varphi| \leq \sqrt{2} \varphi_y \leq 4\sqrt{2}\cos(\pi y)$
on $\Gamma(s)$, as well as
 $y \geq b(s)-(\frac{1}{2} -b(s)) = 2b(s)-\tfrac 12$.  The latter is  because the curve starts at the point $(b(s),b(s))$ and its slope (as a function of $x$) is between 0 and $-1$ on the interval $(b(s),\tfrac 12)$. It follows that 
 \[
 |\nabla \varphi| \leq 4\sqrt{2}  \sin\left( 2 \pi\left( \frac{1}{2}-b(s) \right) \right)
 \]
  on $\Gamma(s)$.  Since $\frac{1}{2}-b(s)\approx 0$ for $s$ near $\frac{2}{\pi}$, the last fraction in \eqref{3.1} is bounded near $\tfrac 2\pi$.

It remains to bound \eqref{eq:bound_T'} for $s$ near $0$.  Since $|\nabla \varphi|, |\nabla \varphi|^{-1}$ are both bounded near $\partial Q$ and the slope of $\Gamma(s)$ is between $0$ and $-1$, we obtain for some $C<\infty$,
\[
\frac 1C |T'_\varphi(s)| \max_{\{\varphi=s\}} |\nabla \varphi|^2 \leq \int_{\Gamma(s)} |\nabla^2 \varphi| d\sigma
\leq  \int_{\Gamma(s)}  \frac{70}{d_c(x,y)} d\sigma 
\leq  \int_{b(s)}^\frac{1}{2} \frac{70\sqrt 2}{x} dx 
\leq 100 |\log b(s)|.
\]
But $|\log b(s)| = |\log (\tfrac{1}{\pi} \arcsin\tfrac {\pi s}2)| \leq |\log \tfrac s 2|\le 2|\log s|$ for $s$ near 0, so (4) is proved.
\end{proof}

\begin{proof}[Proof of Lemma \ref{phi_a}]
Note that $\varphi_a = \varphi$ on $Q\setminus \{d_P<\tfrac 15\}\supseteq B_{3/10}(\frac{1}{2},\frac{1}{2})=:D$. All constants below may depend on $a\in (0,1]$.

The first, second, and fourth claim in (1) follow immediately from Lemma \ref{lemma:phi}(1). We also obtain for $a>0$,
\begin{equation}
\begin{split}
|\nabla \varphi_a(x,y)| & \leq |\nabla \varphi(x,y) | f(d_P(x,y))^{a} +  \varphi(x,y) f(d_P(x,y))^{a-1} \|f'\|_\infty |\nabla d_P(x,y)|
\\  & \leq |\nabla \varphi(x,y) | \|f'\|_\infty^a d_P(x,y)^{a}  +   \varphi(x,y)  \|f'\|_\infty^a d_P(x,y)^{a-1} \leq C d_P(x,y)^{a}
\end{split}
\label{ineq_grad}
\end{equation}
for some $C<\infty$, where in the last inequality we used $\varphi(x,y)\le 2\pi d_P(x,y)$.  So (1) is proved.

(2) is proved similarly: direct differentiation, together with $\varphi(x,y)\le 2\pi d_P(x,y)$,  \eqref{eq:d2phi}, and $|\nabla^2 d_P(x,y)|\le d_P(x,y)^{-1}$ on $\{d_P<\tfrac 15\}$  yield for some $C<\infty$,
\begin{equation}
|\nabla^2 \varphi_a(x,y)| \leq C d_P(x,y)^{a-1}.
\label{d2phia}
\end{equation}

It remains to show (3) and (4). Since $\varphi_a=\varphi$ on $D$ (and $D$ fully contains the level sets $\{ \varphi_a=s\}$ for all $s$ near $\tfrac 2\pi$), all the claims hold when restricted to all $s$ near $\tfrac 2\pi$, due to the same properties of $\varphi$.  We therefore only need to consider $s$ away from $\tfrac 2\pi$. 

The claim in (3) about the level sets of $\varphi_a$ follows from the same statement for $\varphi$, and from positivity of the derivatives of  $\varphi$ and $d_P$ in the direction $(\tfrac 12,\tfrac12)-p$ inside each connected component of $\{d_P<\tfrac 15\}$, where $p$ is the unique point from $P$ belonging to that component.  Also, since $s^{-2a/(a+1)}$ is integrable near 0 for $a\in(0,1)$ and $\nabla \varphi_a\in L^\infty(Q)$ for $a>0$, boundedness of $T_{\varphi_a}$ for $a\in(0,1)$ as well as (4)  for $a\in (0,1]$ will follow if we show $\sup_{s\in(0,s_0]} s^{2a/(a+1)} |T_{\varphi_a}'(s)| <\infty$ for $a\in (0,1]$, with $s_0 := \sup_{(x,y)\in Q\backslash D} \varphi(x,y) <\tfrac 2\pi$.

For any $s\in(0,s_0]$, let $b_a(s)$ be the unique value such that $(b_a(s),b_a(s)) \in \partial Q_T$ (with $Q_T$ from the proof of Lemma \ref{lemma:phi}) and $\varphi_a(b_a(s), b_a(s))=s$. 
A direct computation yields $c' s^{1/(a+1)} \leq b_a(s) \leq C' s^{1/(a+1)} $ for some $c', C'\in(0,\infty)$. Due to \eqref{T''} and symmetry we have 
\begin{equation}
\begin{split}
|T'_{\varphi_a}(s)| \leq 
32 \int_{\{\varphi_a=s\}\cap Q_T} \frac{|\nabla^2 \varphi_a|}{|\nabla \varphi_a|^3} d\sigma 
& = 32 \int_{b_a(s)}^{1/2} \frac{|\nabla^2 \varphi_a (x,y_{a,s}(x))|}{|\nabla \varphi_a (x,y_{a,s}(x))|^2 | (\varphi_a)_y(x,y_{a,s}(x))|} dx 
\\ & \le  32 \int_{b_a(s)}^{1/2} \frac{|\nabla^2 \varphi_a (x,y_{a,s}(x))|}{| (\varphi_a)_y(x,y_{a,s}(x))|^3} dx, 
\end{split}
\label{A11}
\end{equation}
where $y_{a,s}(x)$ is such that $\varphi_a(x,y_{a,s}(x))=s$ and $(x,y_{a,s}(x))\in Q_T$. Its uniqueness is guaranteed by $(\varphi_a)_y> 0$ on $Q_T$, which holds because $\varphi,\partial_y \varphi$, and $f(d_P(x,y))$ are positive on $Q_T$, and $\partial_y[ f(d_P(x,y))]\ge 0$ there. In fact, on $Q_T\cap\{\varphi<s_0\}$ we have for some $C<\infty$,
\[
 (\varphi_a)_y(x,y) \geq \partial_y \varphi (x,y) f(d_P(x,y))^a \geq C d_P(x,y)^a.
\]
Combining this and \eqref{d2phia} with \eqref{A11} yields for $s\in(0,s_0]$ (with a new $C<\infty$),
\[
\begin{split}
\frac 1C |T'_{\varphi_a}(s)| &\leq \int_{b_a(s)}^{1/2} d_P(x,y_{a,s}(x))^{-1-2a} dx\\
&\leq \int_{b_a(s)}^{1/4} x^{-1-2a} dx + \int_{1/4}^{1/2- b_a(s)} \left(x-\frac{1}{2}\right)^{-1-2a}dx + \int_{1/2 - b_a(s)}^{1/2} y_{a,s}(x)^{-1-2a}dx.
\end{split}
\]
The first two integrals are each bounded by $\tfrac 1{2a}(c')^{-2a}s^{-2a/(a+1)}$ (recall that $b_a(s) \geq c' s^{1/(a+1)}$), as we need. For $x \in [\frac{1}{2}-b_a(s), \frac{1}{2}]$ and $s\in(0,s_0]$ we have (with $C:=\|\nabla\varphi\|_\infty \|f'\|_\infty^a$)
\[
\begin{split}
s &= \varphi_a(x,y_{a,s}(x)) \leq C y_{a,s} (x) d_P(x,y_{a,s}(x))^a \leq C(C')^a y_{a,s}(x)(s^{a/(a+1)}+ y_{a,s}(x)^a)
\end{split}
\]
because $\frac{1}{2}-x\le b_a(s)\leq C' s^{1/(a+1)}$. This gives $\max\{y_{a,s}(x)s^{a/(a+1)}, y_{a,s}(x)^{a+1}\}\geq cs$ (with some $c\in (0,1]$), which implies $y_{a,s}(x) \geq c s^{1/(a+1)}$. This and $b_a(s) \leq C' s^{1/(a+1)}$ show that the third integral is bounded by $C'c^{-1-2a} s^{-2a/(a+1)}$ for $s\in(0,s_0]$. Hence we indeed get $\sup_{s\in(0,s_0]} s^{2a/(a+1)} |T_{\varphi_a}'(s)| <\infty$ for $a\in (0,1]$, and we are done.
\end{proof}

\begin{proof}[Proof of Lemma \ref{L.4.4}]
Since $\varphi_{a,\delta} = \varphi_a$ on $D_{a,\delta}\cup\partial Q$, we only need to consider  $s\in(0,\delta)$ in (1).  Fix any $a\in(0,1)$, $\delta\in(0,\tfrac 1{10})$, and $s\in(0,\delta)$.  We also have $\varphi_{a,\delta}(x,y) = \varphi_a(x,y)$ when $d_P(x,y)\ge\tfrac 15 d_{a,\delta}$, so their level sets coincide outside $B_{a,\delta}:=\{(x,y)\in Q : d_P(x,y)< \tfrac 15 d_{a,\delta} \}$.  
(Recall that $d_{a,\delta}=c\delta^{1/(a+1)}$ for some $\delta$-independent $c>0$.) (1)
now follows as the same claim for $\varphi_a$ in the proof of Lemma  \ref{phi_a}(3).

From Lemma \ref{phi_a}(4) and $|\nabla\varphi|^{-1}$  being bounded except near the point $(\tfrac 12,\tfrac 12)$ (and hence $(\sup_{\{\varphi_a=s\}} |\nabla \varphi_a|^2)^{-1}$ being bounded except near $s=\tfrac 2\pi$) it follows that
\[
\sup_{0<s\le\delta<1/10} \delta^{-(1-a)/(a+1)}  | T_{\varphi_{a}}(s)- T_{\varphi_{a}}(\delta)|<\infty.
\]
It is therefore sufficient to show (2) with $T_{\varphi_{a}}(s)$ in place of $T_{\varphi_{a}}(\delta)$.
The parts of the integrals defining $T_{\varphi_a}(s)$ and $T_{\varphi_{a,\delta}}(s)$ coincide outside $B_{a,\delta}$, so
\[
T_{\varphi_{a,\delta}}(s)- T_{\varphi_{a}}(s) = \int_{\{\varphi_{a,\delta}=s\}\cap B_{a,\delta}} \frac 1{|\nabla \varphi_{a,\delta}|} d\sigma - \int_{\{\varphi_{a}=s\}\cap B_{a,\delta}}  \frac 1{|\nabla \varphi_{a}|} d\sigma. 
\]
It therefore suffices to show
\beq\lb{A.10}
\sup_{0<s\le\delta<1/10} \left( \delta^{(1-a)/(a+1)} \log \tfrac{2\delta}s \right)^{-1} \max\left\{ \int_{\{\varphi_{a}=s\}\cap B_{a,\delta}}  \frac 1{|\nabla \varphi_{a}|} d\sigma,  \int_{\{\varphi_{a,\delta}=s\}\cap B_{a,\delta}} \frac 1{|\nabla \varphi_{a,\delta}|} d\sigma  \right\} <\infty.
\eeq

On $B_{a,\delta}$ we have $d_P\le\tfrac 1{10}$ (because $d_{a,\delta}<\tfrac 12$, which is due to $\varphi_a(\tfrac 14,\tfrac 14)=\tfrac{\sqrt 2}\pi>\delta$), so  
\[
\varphi_a(x,y)=5^a \varphi(x,y)d_P(x,y)^a \qquad \text{and}\qquad  \varphi_{a,\delta}(x,y)=5^a \varphi(x,y)d_P(x,y)^a f \left( \frac{d_P(x,y)}{d_{a,\delta}} \right)^{1-a}
\]
there.  Also, $B_{a,\delta}$ has 8 connected components, but the following analysis is essentially the same in each of them.  We will therefore only consider the one near the origin (or rather one half of it, due to symmetry).  So let $B':=B_{d_{a,\delta}/5}(0,0)\cap \{ 0<y<x\}$.

On $B'$ we have $\varphi_x,\varphi_y,(d_P)_x,(d_P)_y>0$, which together with $f'\ge 0$ shows that (with the usual $\sim$ notation, where constants depend on $a$ but not on $\delta,s$)
\beq\lb{A.11}
|\nabla\varphi_{a}|\sim (\varphi_{a})_x+ (\varphi_{a})_y  \qquad\text{and}\qquad  |\nabla\varphi_{a,\delta}|\sim (\varphi_{a,\delta})_x+ (\varphi_{a,\delta})_y 
\eeq
and also that the curves $\{\varphi_{a}=s\}\cap B'$ and $\{\varphi_{a,\delta}=s\}\cap B'$ are graphs of decreasing functions of $x$.  Those graphs start at some points $(b_{a}(s),b_{a}(s))$ and $(b_{a,\delta}(s),b_{a,\delta}(s))$  (the former being from the proof of Lemma \ref{phi_a}).  Since $\varphi(x,y)\sim y$ and $d_P(x,y)\sim x$ on $B'$, and $f(r)\sim r$ on $(0,1)$, we obtain (using also $d_{a,\delta}=c \delta^{1/(a+1)}$)
\beq\lb{A.12}
b_{a}(s)\sim s^{1/(a+1)} \qquad\text{and}\qquad b_{a,\delta}(s)\sim \delta^{(1-a)/(2a+2)} s^{1/2}  .
\eeq
Since on $B'$ we also have $0\le \varphi_x\le \varphi_y \sim 1$ and $0\le (d_P)_y\le (d_P)_x\sim 1$, it follows that
\[
|\nabla\varphi_a(x,y)|\sim x^a \qquad\text{and}\qquad |\nabla\varphi_{a,\delta}(x,y)|\sim \delta^{-(1-a)/(a+1)} x
\]
on $B'$.  From this and \eqref{A.12} we obtain that the first integral in \eqref{A.10}, with $B_{a,\delta}$ replaced by $B'$, is bounded above by  a constant times
\[
\int_{b_a(s)}^{d_{a,\delta}} x^{-a} (1+|h'(x)|)dx\le  \int_{b_a(s)}^{d_{a,\delta}} x^{-a} dx + b_a(s)^{-a} b_a(s) \lesssim \delta^{(1-a)/(a+1)},
\]
with $h$ the decreasing function whose graph is $\{\varphi_{a}=s\}\cap B'$.
By the same argument (and also using $\tfrac 1{1+a}-\tfrac {1-a}{2a+2}=\tfrac 12$), the second integral is bounded above by a constant times
\[
\int_{b_{a,\delta}(s)}^{d_{a,\delta}} \delta^{(1-a)/(a+1)} x^{-1} dx + \delta^{(1-a)/(a+1)}b_{a,\delta}(s)^{-1} b_{a,\delta}(s)\lesssim \delta^{(1-a)/(a+1)}(1+\log \tfrac\delta s ).
\]

The same bounds are obtained for the integrals from \eqref{A.10} over the remaining parts of $B_{a,\delta}$ (the 4 connected components in the corners of $Q$ are divided into 2 parts each, the other 4 connected components into 4 parts each).  This proves \eqref{A.10}, and thus (2).
%
\end{proof}

\begin{lemma} \lb{L.A.2}
There is $c>0$ such that for any $s\in [0,1]$ and $\kappa,\eps \in (0,\tfrac 12]$, a function $f\in L^\infty(Q)$ is $\kappa$-mixed to scale $\eps$ whenever $\|f\|_{{H}^{-s}}\leq c\kappa^{3s/2}\eps^{1+s} \|f\|_\infty $.
\end{lemma}

\begin{proof}
Let us consider the same test function $g:\mathbb{R}^2 \to [0,1]$ as in \cite[Lemma 2.3]{IKX},  smooth and satisfying $g=1$ on $B(0,\eps)$, $g=0$ on $\bbR^2\setminus B(0,(1+\frac{\kappa}{20})\eps)$, and $\|\nabla g\|_\infty \leq \tfrac{40}{\kappa \eps}$. Elementary computations yield $\|g\|_{2} \leq C_1\eps$ and $\|g\|_{\dot{H}^1} \leq C_1\kappa^{-1/2}$ for some $C_1<\infty$. Interpolation then gives us $\|g\|_{\dot{H}^s} \leq  C_1 \kappa^{-s/2} \eps^{1-s}$ for $s\in[0,1]$.
We thus have for any $x\in Q$ such that $d(x,\partial Q)\geq  2\eps$ (so that $g(\cdot-x)$ is supported on $Q$),
\[
\left|\int_Q  f (y)g(y-x)dy\right| \leq \| f \|_{{H}^{-s}} \|g\|_{\dot{H}^s}\leq C_1c \kappa \eps^2\| f \|_\infty ,
\] 
which gives
\begin{equation}
\left|\fint_{B_\eps(x)}  f (y)dy\right| \leq \frac{1}{\pi \eps^2}\left( \left|\int_Q  f (y)g(y-x)dy\right| + \frac{3\pi \kappa \eps^2}{20} \| f \|_\infty \right)
\leq \left( \frac{C_1c}\pi  + \frac{3}{20}\right) \kappa \| f \|_\infty.
\label{avr1}
\end{equation}

For $x\in Q$ such that $d(x,\partial Q) <2\eps$ we instead use the test function $g(\cdot-x) h(\cdot)$, where $h(y)=1$ when $d(y, \partial Q) \geq \frac{\kappa \eps}{20}$, $h=0$ on $\bbR^2\setminus Q$ and $\|\nabla h\|_{\infty} \leq \frac{40}{\kappa \eps}$. An identical argument as above yields $\|g(\cdot-x) h(\cdot)\|_{\dot{H}^s} \leq C_2 \kappa^{-s/2} \eps^{1-s} $ for some $C_2<\infty$ and all $s\in[0,1]$.  Similarly to \eqref{avr1} we also obtain
\begin{equation}
\left|\fint_{B_\eps(x)\cap Q}  f (y) dy\right| \leq \left( \frac{4C_2c}\pi  + \frac{12}{20}\right) \kappa \| f \|_\infty
\label{avr2}
\end{equation}
because $|B_\eps(x)\cap Q|\ge\tfrac\pi 4 \eps^2$.
Choosing $c:=\tfrac\pi{10}\max\{C_1,C_2\}^{-1}$ finishes the proof.
\end{proof}

\begin{lemma} \label{sobolev_1}
Let $p\in[1,\infty)$,  $s\in(0,\tfrac 1p)$, and $n\ge 0$. Assume that a family of functions $\{f_{ij}\}_{i,j=0}^{2^n-1}$ on $Q$ satisfies $\|f_{ij}\|_{W^{s,p}} \leq 1$ and $\|f_{ij}\|_{\infty}\leq 1$ for any $i,j \in \{0,\ldots,2^n-1\}$. 
Let $f:Q\to \mathbb{R}$ be defined on each $Q_{nij}$ by
\[
f(x) = 2^{-n} f_{ij}(2^n x- ( i, j)).
\]
Then we have $\|f\|_{W^{s,p}} \leq C_{s,p} \, 2^{-(1-s)n}$ for some $C_{s,p}>0$ which depends only on $s,p$.
\end{lemma}

\begin{proof}
Since $s\in(0,1)$ and $Q\subseteq \mathbb{R}^2$, the fractional Sobolev norm can equivalently be defined by  (see, e.g., \cite{NPV})
\[
\|f\|_{W^{s,p}} = \Bigg( \int_Q |f|^p dx + \underbrace{\int_{Q^2} \frac{|f(x)-f(y)|^p}{|x-y|^{2+sp}} dxdy}_{=: I[f]} \Bigg)^{1/p}.
\]
Obviously $\|f\|_{p} \leq 2^{-n}$, so it is sufficient to show  $I[f] \leq C_{s,p} 2^{-(1-s)np}$. 
This would follow from 
\begin{equation}\label{goal}
 \int_{Q_{nij}\times Q}  \frac{|f(x)-f(y)|^p}{|x-y|^{2+sp}} dxdy \leq C_{s,p} \,2^{-2n-(1-s)np} 
\end{equation}
for any $i,j \in\{0,\ldots,2^n-1\}$, which we shall now prove.

Since $I[f_{ij}] \leq 1$ for any $i,j$ by the hypothesis, we have
\begin{equation}
\int_{Q_{nij}^2}  \frac{|f(x)-f(y)|^p}{|x-y|^{2+sp}} dxdy = 2^{-2n-(1-s)np} I[f_{ij}] \leq  2^{-2n-(1-s)np}.
\label{self}
\end{equation}
Also, for some $C_{s,p}<\infty$ and any $x\in Q_{nij}$ we have using $\|f\|_{L^\infty(Q)} \leq 2^{-n}$,
\[
\int_{Q\setminus Q_{nij}} \frac{|f(x)-f(y)|^p}{|x-y|^{2+sp}}dy \leq \int_{d(x,\partial Q_{nij})}^{\sqrt 2}  \frac{(2\|f\|_\infty)^p}{r^{2+sp}}  2\pi rdr  \leq C_{s,p} \,2^{-np} d(x,\partial Q_{nij})^{-sp}.
\]
Integrating this in $x\in Q_{nij}$ and using $sp<1$ gives (with a new $C_{s,p}$)
\[
\int_{Q_{nij} \times (Q\setminus Q_{nij})} \frac{|f(x)-f(y)|^p}{|x-y|^{2+sp}}dxdy \leq C_{s,p}\,  2^{-np} \int_{Q_{nij}} d(x,\partial Q_{nij})^{-sp}dx \leq C_{s,p}\, 2^{-2n-(1-s)np}.
\]
 Adding this to \eqref{self} yields \eqref{goal}.
\end{proof}



\end{document}